\pdfoutput=1
\documentclass[11pt,letterpaper]{article}
\usepackage[scale=0.78]{geometry}
\usepackage[T1]{fontenc}
\usepackage{lmodern}
\usepackage[utf8]{inputenc}
\usepackage{amsmath,amssymb,amsthm,mathtools}
\usepackage{graphicx}
\usepackage{booktabs}
\usepackage{caption}
\usepackage{xcolor}
\usepackage[colorlinks,linkcolor=blue!50!black,citecolor=blue!50!black,urlcolor=blue!50!black]{hyperref}
\hypersetup{
  pdftitle={Integrable building blocks for symplectic neural networks: learned-type layers, fiber-preserving layers, and what a single layer can do},
  pdfauthor={Andrew Qing He},
}

\newtheorem{lemma}{Lemma}
\newtheorem{proposition}{Proposition}
\newtheorem{corollary}{Corollary}
\theoremstyle{definition}

\newtheorem{example}{Example}
\theoremstyle{remark}
\newtheorem{remark}{Remark}

\newcommand{\R}{\mathbb{R}}
\newcommand{\dd}{\mathrm{d}}
\newcommand{\T}{\top}
\DeclareMathOperator{\Sp}{Sp}

\title{Integrable building blocks for symplectic neural networks:\\
learned-type layers, fiber-preserving layers, and what a single layer can do}
\author{Andrew Qing He\\[3pt]
 \normalsize Duke University\\
 \normalsize\texttt{andrew.hq@duke.edu}}
\date{}

\begin{document}
\maketitle

\begin{abstract}
Exactly symplectic neural networks are built almost exclusively from shear maps, and gain
expressivity only by stacking more of them. We ask what else a single explicit symplectic
layer can be. Starting from the observation that every layer in use is the exact flow of a
Hamiltonian that is a function of commuting integrals, we obtain three families of layers:
a \emph{learned-type layer}, in which every conjugate pair is rotated, squeezed or sheared by
a state-dependent amount and the type is found by training; \emph{moving-plane shears}, which
are residual symplectic maps outside the known ridge-function class; and \emph{fiber-preserving
layers}, which combine an unconstrained point transformation with a fiber translation at the
cost of one linear solve. We characterize residual symplectic layers $x\mapsto x+J\nabla H(x)$:
in one degree of freedom they are exactly the shears along a line, while from two degrees of
freedom on a strictly larger class exists. Conjugating a learned-type layer by a symplectic
network yields an exact one-parameter group with an exactly conserved learned Hamiltonian.
In experiments on sixteen systems with one to forty degrees of freedom, against SympNets,
H\'enonNets and learned-plane shear networks, layers built on nonlinear integrals are far
more accurate when the Hamiltonian contains functions of those integrals. On nonlinear
Schr\"odinger lattices with six to forty sites, twist layers alternating with linear
symplectic maps are more than a hundred times more accurate than every shear-based network;
with an on-site potential of learned form they are three to four orders of magnitude more
accurate, and also more accurate than a split-step network with the correct fixed
nonlinearity. On chains of Tao oscillators with the step size as an input, stretch layers
are forty-five to seventy times more accurate than learned-plane shears and recover the
Hamiltonian to one percent. Conjugated flows remain accurate at thirty times the training
horizon in one degree of freedom. On generic non-separable systems the gains are small:
fiber layers improve on learned-plane shears by at most a factor of two, and widening the
networks does not help any architecture.
\end{abstract}



\section{Introduction}

Let $x=(q,p)\in\R^{2n}$ and $J=\begin{psmallmatrix}0&I\\-I&0\end{psmallmatrix}$. A map $\varphi$
is symplectic if $(\dd\varphi)^\T J\,\dd\varphi=J$; flows of Hamiltonian systems
$\dot x=J\nabla H(x)$ are symplectic. Networks that are symplectic by construction
\cite{jin2020sympnets,burby2021henonnets,tapley2024symplectic} learn such flow maps from data
without a penalty term and retain the qualitative long-time behavior of the dynamics.

Nearly all of these networks are compositions of \emph{shears},
\begin{equation}\label{eq:shears}
 (q,p)\mapsto(q,\,p+\nabla F(q)),\qquad (q,p)\mapsto(q+\nabla G(p),\,p),
\end{equation}
or of shears along other Lagrangian planes \cite{feng1998variations,tapley2024symplectic}.
A single shear moves only half of the variables, and, as we confirm experimentally
(Section~\ref{sec:experiments}), enlarging the scalar networks $F,G$ at fixed depth does not
improve accuracy: these architectures gain expressivity through depth alone. This paper asks
what a single explicit symplectic layer can be, beyond a shear.

The starting point of this work was the \emph{stretch} map of \cite{he2024deep},
\begin{equation}\label{eq:stretch}
 Q=e^{\nabla\Phi(q\odot p)}\odot q,\qquad P=e^{-\nabla\Phi(q\odot p)}\odot p,
\end{equation}
obtained as a symplectic counterpart of the scaling in real NVP coupling layers
\cite{dinh2017density}. As a map it is classical: for $n=1$ it is Moser's normal form of an
area-preserving map at a hyperbolic fixed point \cite{moser1956analytic}, and it is the exact
flow of the Hamiltonian $\Phi(q_1p_1,\dots,q_np_n)$. That observation is the organizing
principle of the paper.

\paragraph{Contributions.}
\begin{enumerate}
 \item \emph{A calculus} (Section~\ref{sec:calculus}). Shears, the stretch map and the twist
 map are flows of $\Phi(I_1,\dots,I_k)$ for commuting integrals $I_j$ with explicit flows.
 We use this classical fact, together with nested integrals and cotangent lifts, to derive
 new layers and to locate existing ones.
 \item \emph{Learned-type layers} (Section~\ref{sec:quad}): trainable quadratic integrals per
 conjugate pair; the layer is a state-dependent rotation, squeeze or shear, with the type
 learned.
 \item \emph{Residual symplectic layers} (Section~\ref{sec:residual}): $x\mapsto x+J\nabla H(x)$
 is symplectic iff $H''JH''\equiv0$. For $n=1$ these are exactly the ridge functions; for
 $n\ge2$ we give solutions that are not functions of commuting linear forms, and the
 corresponding \emph{moving-plane} layer.
 \item \emph{Fiber-preserving layers} (Section~\ref{sec:fiber}): $Q=f(q)$,
 $P=Df(q)^{-\T}(p+\nabla F(q))$ with $f$ an unconstrained residual network. One such layer and
 its dual represent a generic linear symplectic map exactly, where shears need four or five
 factors.
 \item \emph{Conjugated integrable flows and step-size-aware networks}
 (Section~\ref{sec:time}).
 \item \emph{Experiments} (Section~\ref{sec:experiments}) on sixteen systems, against SympNets,
 H\'enonNets and learned-plane shear networks, reporting also where the new layers lose. The
 main findings: twist and stretch layers are specialists that win by one to two orders of
 magnitude on Hamiltonians containing their integral, up to forty degrees of freedom; the
 learned-type layer is close to the better specialist in each case without being told the
 integral; alternating twist layers with linear symplectic maps gives an exactly symplectic
 learned split-step scheme, whose accuracy depends strongly on how the potential is
 parametrized (general, on-site of learned form, or fixed form); and widening the potentials
 does not help any architecture.
\end{enumerate}

\paragraph{Related work.}
SympNets \cite{jin2020sympnets} compose linear, activation and gradient modules, all shears,
and are universal approximators of symplectic maps; H\'enonNets \cite{burby2021henonnets} use
H\'enon maps and rely on \cite{turaev2003polynomial}. Tapley \cite{tapley2024symplectic}
builds layers as exact flows of ridge Hamiltonians $N(Ap+Bq)$ with $AB^\T$ symmetric, i.e.\
shears along learned Lagrangian planes, following \cite{feng1998variations}; SympFlow
\cite{canizares2024symplectic} uses flows of time-dependent separable Hamiltonians. Layers as
exact flows of simple Hamiltonians are thus an established design principle
\cite{tapley2024symplectic,canizares2024symplectic,eldred2023lie}, which we do not claim; what
we add is the use of functions of commuting \emph{nonlinear} integrals.
Point transformations $Q=f(q)$, $P=Df^{-\T}p$ were introduced as symplectic layers with
real NVP $f$ by Li et al.\ \cite{li2020neural}; an unconstrained network with an explicitly
inverted Jacobian was used on angle variables in \cite{ibata2021actionfinder}.
Rotations of a latent pair by an action-dependent angle appear as the time-evolution step of
Koopman and action-angle models
\cite{lusch2018deep,bondesan2019learning,ishikawa2021neural,daigavane2022learning}; there the
integral is fixed and the map is used once, in latent space.
Janik and Benner \cite{janik2025time} show that step-size-scaled $q$/$p$-shear networks have
separable generators and leave non-separable systems open.
Implicit generating-function networks \cite{chen2021data} and Hamiltonian neural networks
\cite{greydanus2019hamiltonian} are outside our scope: we consider only explicit maps.

\section{A calculus of explicit symplectic layers}\label{sec:calculus}

\subsection{Functions of commuting integrals}

\begin{lemma}[classical; see e.g.\ {\cite[\S49]{arnold1989mathematical}}]\label{lem:commuting}
Let $I_1,\dots,I_k\in C^2(\R^{2n})$ Poisson-commute and have complete flows $\varphi^j_t$. For
$\Phi\in C^2(\R^k)$ the flow of $H=\Phi(I_1,\dots,I_k)$ is
\[
 \varphi^H_t(x)=\varphi^1_{t\,\theta_1(x)}\circ\dots\circ\varphi^k_{t\,\theta_k(x)}(x),
 \qquad \theta_j(x)=\frac{\partial\Phi}{\partial I_j}\bigl(I(x)\bigr).
\]
\end{lemma}
\begin{proof}
$X_H=\sum_j\theta_j X_{I_j}$, and $\{I_i,H\}=\sum_j\theta_j\{I_i,I_j\}=0$, so every $I_i$, hence
every $\theta_j$, is constant along the flow of $H$ and along each $\varphi^j$. The commuting
flows therefore compose with the frozen times $t\theta_j(x)$.
\end{proof}

A \emph{layer} is the map $\varphi^H_1$ with $\Phi$ a scalar network. It is symplectic, and
its inverse is the layer with $\Phi$ replaced by $-\Phi$, since the $I_j$ are invariant.
Evaluating it costs one gradient of $\Phi$ and the explicit flows of the $I_j$.

\begin{example}
(i) $I=q$: the $q$-shear in \eqref{eq:shears}, with $\Phi=-F$. (ii) $I=p$: the $p$-shear.
(iii) $I=Cx$ with $CJC^\T=0$: shears along a Lagrangian plane
\cite{feng1998variations,tapley2024symplectic}. (iv) $I_j=q_jp_j$: the stretch map
\eqref{eq:stretch}. (v) $I_j=\tfrac12(q_j^2+p_j^2)$: the twist map, rotating the $j$-th pair
by the angle $\partial\Phi/\partial I_j$. (vi) For $n\ge2$, $I_1=q_1p_1+q_2p_2$ and
$I_2=q_1p_2-q_2p_1$: a joint dilation and rotation (focus--focus type).
\end{example}

Cases (iv)--(vi) are, up to linear symplectic changes of variables, the nondegenerate types
in Williamson's classification of quadratic Hamiltonians \cite{williamson1936algebraic}.

\subsection{The learned-type layer}\label{sec:quad}

Write $x_j=(q_j,p_j)$ and give every pair a trainable quadratic-affine integral
\begin{equation}\label{eq:quad-invariant}
 I_j(x_j)=\tfrac12 x_j^\T M_jx_j+c_j^\T x_j,\qquad
 M_j=\begin{pmatrix}a_j&b_j\\ b_j&d_j\end{pmatrix},\quad c_j\in\R^2 .
\end{equation}
Integrals of different pairs commute, so Lemma~\ref{lem:commuting} applies with $k=n$. The
flow of $I_j$ solves $\dot x_j=A_jx_j+J_2c_j$ with $A_j=J_2M_j$, and $A_j^2=-\Delta_jI$,
$\Delta_j=\det M_j$. With the entire functions
\[
 C(u)=\cos\sqrt u,\qquad S(u)=\frac{\sin\sqrt u}{\sqrt u},\qquad R(u)=\frac{1-\cos\sqrt u}{u}
\]
(continued to $u<0$ by $\cosh$, $\sinh$), the layer is
\begin{equation}\label{eq:quad-layer}
 x_j\mapsto\bigl(C(u_j)I+\theta_jS(u_j)A_j\bigr)x_j+\bigl(\theta_jS(u_j)I+\theta_j^2R(u_j)A_j\bigr)J_2c_j,
 \qquad u_j=\Delta_j\theta_j^2,\quad \theta=\nabla\Phi(I(x)).
\end{equation}
It rotates the pair if $\Delta_j>0$, squeezes it hyperbolically if $\Delta_j<0$ and shears it
if $\Delta_j=0$, about a center determined by $c_j$, and depends smoothly on $(M_j,c_j)$
across the three cases. The parameters $M_j,c_j$ are trained together with $\Phi$, so the
type of each degree of freedom is not fixed in advance. The stretch map is
$M_j=\begin{psmallmatrix}0&1\\1&0\end{psmallmatrix}$, $c_j=0$; the twist map is $M_j=I$; a
shear in $q_j$ alone is $M_j=0$, $c_j=e_1$. Pairs interact through $\Phi$ only.

\subsection{Nested integrals and moving-plane shears}\label{sec:moving}

The integrals in Lemma~\ref{lem:commuting} may depend on one another. Split the degrees of
freedom into two groups $A$ and $B$ and put $a=q_A$, $y_B=(q_B,p_B)$. For a symmetric
matrix-valued network $S(a)$ the linear forms
\[
 u=q_B+S(a)\,p_B
\]
span, for each fixed $a$, a Lagrangian plane of the $B$-block; the components of $a$ and $u$
commute pairwise. The flow of $H=\Phi(a,u)$ is
\begin{equation}\label{eq:moving}
\begin{aligned}
 q_B&\mapsto q_B+t\,S(a)\Phi_u, \qquad p_B\mapsto p_B-t\,\Phi_u,\\
 p_A&\mapsto p_A-t\,\nabla_a\bigl[\Phi(a,q_B+S(a)p_B)\bigr]
      +\tfrac{t^2}{2}\,\nabla_a\bigl[\phi^\T S(a)\phi\bigr]_{\phi=\Phi_u},
\end{aligned}
\end{equation}
with $q_A$ unchanged; the last gradient is taken with $\phi$ held fixed. We call this a
\emph{moving-plane shear}: block $B$ is sheared along a plane that turns with $q_A$, and
$p_A$ is shifted to compensate. This is the symplectic form of conditioning in a coupling
layer \cite{dinh2017density}: any family of symplectic maps of block $B$ parametrized by
$q_A$ extends to a symplectic map of the whole space by a shift of $p_A$.
Because the plane moves, $H$ is not a function of commuting linear forms
(Section~\ref{sec:residual}).

\section{Residual symplectic layers}\label{sec:residual}

The simplest way to add many branches in parallel is a residual update $x\mapsto x+\tau J\nabla H(x)$.

\begin{proposition}\label{prop:residual}
Let $H\in C^2(\R^{2n})$. The map $x\mapsto x+\tau J\nabla H(x)$ is symplectic for all $\tau$
(equivalently, for one $\tau\ne0$) if and only if
\begin{equation}\label{eq:nilpotent}
 H''(x)\,J\,H''(x)=0\qquad\text{for all }x,
\end{equation}
that is, iff the image of the Hessian is everywhere isotropic.
\end{proposition}
\begin{proof}
$(I+\tau JH'')^\T J(I+\tau JH'')=J+\tau(H''J^\T J+J^2H'')+\tau^2H''J^\T JJH''=J+\tau^2H''JH''$,
using $J^\T J=I$ and $J^2=-I$.
\end{proof}

Condition \eqref{eq:nilpotent} holds for $H=\phi(Cx)+b^\T x$ with $CJC^\T=0$
\cite{feng1998variations}: these are the shears along Lagrangian planes.

\begin{corollary}[one degree of freedom]\label{cor:n1}
For $n=1$, \eqref{eq:nilpotent} is $\det H''\equiv0$, and a $C^2$ solution on all of $\R^2$ has
the form $H(x)=g(a^\T x)+b^\T x+c$. Hence a globally defined residual symplectic layer in one
degree of freedom is a shear along a fixed line, composed with a translation.
\end{corollary}
\begin{proof}
For a symmetric $2\times2$ matrix $A$, $AJA=\det(A)J$. Entire $C^2$ solutions of the
homogeneous Monge--Amp\`ere equation in the plane are cylinders by the theorem of Hartman and
Nirenberg \cite{hartman1959spherical}.
\end{proof}

\begin{proposition}[two or more degrees of freedom]\label{prop:counter}
For $n\ge2$ there are smooth solutions of \eqref{eq:nilpotent} on $\R^{2n}$ that are not of
the form $\phi(Cx)+b^\T x$ with $CJC^\T=0$. For $n=2$,
\[
 H(q,p)=q_2\cos q_1+p_2\sin q_1
\]
is one: the image of $H''(x)$ is spanned by $e_{q_1}$ and $-\sin q_1\,e_{q_2}+\cos q_1\,e_{p_2}$,
which is isotropic for every $x$, while the images at different $x$ together span a
three-dimensional, hence non-isotropic, subspace.
\end{proposition}
The example is the symplectic reading of Sacksteder's complete non-cylindrical hypersurface
with Gauss map of rank two \cite{sacksteder1960hypersurfaces}. More generally
$H=F(q_A)+c(q_A)^\T y_B$ satisfies \eqref{eq:nilpotent} whenever $c^*\omega_B=0$, in particular
for every curve $c$ when $|A|=1$. These Hamiltonians are instances of the nested integrals of
Section~\ref{sec:moving}. A classification of the solutions of \eqref{eq:nilpotent} for
$n\ge2$ is, to our knowledge, open.

\begin{remark}
A single explicit layer whose accuracy grows with width, for arbitrary targets, would amount
to an explicit symplectic one-step method for general non-separable Hamiltonians, which is
not known \cite{mclachlan2002splitting,hairer2006geometric}. Corollary~\ref{cor:n1} shows that
the residual form cannot provide one in one degree of freedom.
\end{remark}

\section{Fiber-preserving layers}\label{sec:fiber}

For a diffeomorphism $f$ of $\R^n$ and $F:\R^n\to\R$,
\begin{equation}\label{eq:fiber}
 Q=f(q),\qquad P=Df(q)^{-\T}\bigl(p+\nabla F(q)\bigr)
\end{equation}
is symplectic, with generating function $S(q,P)=f(q)^\T P-F(q)$; these are the symplectic
maps preserving the fibration $(q,p)\mapsto q$. Li et al.\ \cite{li2020neural} take $f$ to be
an explicitly invertible flow and obtain $P$ by a vector--Jacobian product of $f^{-1}$.
We instead take
\[
 f(q)=q+B(q)
\]
with $B$ an unconstrained dense network. Since $S$ is linear in $P$, the forward map is still
explicit: one Jacobian and one $n\times n$ linear solve. The inverse map requires $f^{-1}$ and
is not explicit; the layer is symplectic wherever $Df$ is invertible. The dual layer acts on
$p$: $P=g(p)$, $Q=Dg(p)^{-\T}(q-\nabla G(p))$. We write F and G for the two layers.

\begin{proposition}\label{prop:linear}
Let $M=\begin{psmallmatrix}A&B\\C&D\end{psmallmatrix}\in\Sp(2n)$ with $A$ invertible. Then
\[
 M=\begin{pmatrix}I&0\\CA^{-1}&I\end{pmatrix}\begin{pmatrix}A&B\\0&A^{-\T}\end{pmatrix},
\]
where $CA^{-1}$ is symmetric. The second factor is a linear G layer and the first a $q$-shear,
so a generic linear symplectic map is one G layer followed by one F layer. Compositions of
unit triangular symplectic matrices need up to five factors \cite{jin2022optimal}.
\end{proposition}
\begin{proof}
$M^\T JM=J$ gives $A^\T C=C^\T A$ and $A^\T D-C^\T B=I$. The first identity is the symmetry of
$CA^{-1}$; the second gives $D-CA^{-1}B=A^{-\T}$, which is the lower right block of the product.
The second factor is $P=A^{-\T}p$, $Q=A(q-\nabla G(p))$ with $G(p)=-\tfrac12p^\T A^{-1}Bp$, and
$A^{-1}B$ is symmetric because the factor is symplectic.
\end{proof}

Shears are the special cases $f=\mathrm{id}$ of F and G and $M_j=0$ of the learned-type layer,
so networks of either kind inherit the universal approximation property of shear networks
\cite{turaev2003polynomial,jin2020sympnets}.

\section{Time as an input}\label{sec:time}

\paragraph{Step-size-aware networks.}
Replacing each layer $\varphi^{H_k}_1$ by $\varphi^{H_k}_h$ gives a family
$\psi_h=\varphi^{H_K}_h\circ\dots\circ\varphi^{H_1}_h$ with $\psi_0=\mathrm{id}$ and
$\partial_h\psi_h|_{h=0}=J\nabla\sum_kH_k$: a splitting integrator for the learned Hamiltonian
$\sum_kH_k$. With $q$/$p$ shears alone the generator is separable \cite{janik2025time}; with
stretch layers it is $K(p)+V(q)+W(q\odot p)$, and with learned-type or learned-plane layers a
sum of functions of quadratic or linear forms. Tao's test Hamiltonian \cite{tao2016explicit}
$\tfrac12(q^2+1)(p^2+1)=\tfrac12q^2+\tfrac12p^2+\tfrac12(qp)^2+\tfrac12$ lies in the second
class, and is therefore integrated by an explicit splitting method that is symplectic in the
original phase space.

\paragraph{Conjugated integrable flows.}
For a symplectic network $\psi$ and a single layer with integrals $I$ and potential $\Phi$,
\begin{equation}\label{eq:conj}
 x(t)=\psi^{-1}\bigl(\varphi^{\Phi\circ I}_t(\psi(x_0))\bigr)
\end{equation}
is an exact one-parameter group of symplectic maps, the flow of $H=\Phi\circ I\circ\psi$,
which it conserves exactly. With a twist layer this is an action-angle model in Cartesian
form \cite{bondesan2019learning,daigavane2022learning}; with a learned-type layer the model
need not be told whether each degree of freedom is elliptic or hyperbolic.
Model \eqref{eq:conj} can only represent integrable dynamics.

\section{Experiments}\label{sec:experiments}

\paragraph{Setup.}
Targets are flow maps of Hamiltonian systems with one to forty degrees of freedom
(Table~\ref{tab:systems}); reference solutions are computed to tolerance $10^{-12}$. Models
are trained on 1000 uniformly sampled initial conditions (2000 for $n\ge3$) by full-batch Adam
(5000 epochs, learning rate $10^{-2}\to10^{-4}$) followed by at most 500 L-BFGS iterations, in
double precision, and evaluated on 4000 fresh points. All scalar potentials are dense networks
with hidden sizes $(20,10)$ ($(40,20)$ for $n\ge3$) and $\tanh$ activation. We report medians
over seeds; models are matched in parameter count within about $20\%$ unless stated.
Layer codes: Q, P are the shears \eqref{eq:shears}; S the stretch map \eqref{eq:stretch};
T the twist map; W the learned-type layer \eqref{eq:quad-layer}; Z a shear along a learned
Lagrangian plane in the parametrization of \cite{tapley2024symplectic}; M the moving-plane
shear \eqref{eq:moving}; F, G the fiber layers \eqref{eq:fiber}; L a linear symplectic map
(four unit triangular factors and a translation). A string such as
\texttt{PQWPQ} lists the layers in order of application. G-SympNets, LA-SympNets
\cite{jin2020sympnets} and H\'enonNets \cite{burby2021henonnets} are our own
re-implementations.

\begin{table}[t]
\centering\footnotesize
\caption{Test systems. ``In class'' marks Hamiltonians of the form $V(q)+K(p)+W(q\odot p)$.}
\label{tab:systems}
\begin{tabular}{llll}
\toprule
system & $n$ & Hamiltonian & remarks\\
\midrule
pendulum & 1 & $p^2/2-\cos q$ & separable\\
saddle & 1 & $p^2/2-\log\cosh q$ & separable, hyperbolic\\
double well & 1 & $p^2/2-q^2/2+q^4/4$ & separable, saddle and two centers\\
Tao & 1 & $(q^2+1)(p^2+1)/2$ & non-separable, in class\\
H\'enon--Heiles & 2 & standard & separable\\
coupled Tao & 2 & $\sum_i(q_i^2+1)(p_i^2+1)/2+q_1p_1q_2p_2/2$ & non-separable, in class\\
Tao chain & 6, 10 & $\sum_i(q_i^2+1)(p_i^2+1)/2+q_ip_iq_{i+1}p_{i+1}/2$ & non-separable, in class\\
charged particle & 2 & $|p-A(q)|^2/2$, $B=1+|q|^2$ & non-separable\\
double pendulum & 2 & standard & non-separable\\
three-link pendulum & 3 & $p^\T M(q)^{-1}p/2+V(q)$ & non-separable, chaotic\\
two charged particles & 4 & as above, soft Coulomb coupling & non-separable\\
nonlinear Schr\"odinger lattice & 6--40 & $\sum_i\tfrac12|\psi_{i+1}-\psi_i|^2+\tfrac14|\psi_i|^4$ & non-separable\\
saturable lattice & 10, 20 & $\sum_i\tfrac12|\psi_{i+1}-\psi_i|^2+\log(1+|\psi_i|^2)$ & non-separable\\
FPUT-$\beta$ chain & 8 & standard & separable\\
saddle--center & 2 & $\Phi_0(\tfrac12(Q_1^2+P_1^2),Q_2P_2)\circ\psi_0$ & integrable, mixed type\\
periodic Toda lattice & 3 & $\sum_ip_i^2/2+e^{q_i-q_{i+1}}$ & integrable, one free pair\\
\bottomrule
\end{tabular}
\end{table}

\subsection{Training and convergence}\label{sec:convergence}

Figure~\ref{fig:convergence} shows the training loss per epoch and the test error every 250
epochs for one representative run of each model on nine systems covering every experiment
group. Three features recur. The Adam phase brings every model to a plateau that is set by
the architecture: on the same data, models with the right integrals reach plateaus one to
three orders of magnitude lower, and they are close to them within the first thousand
epochs. L-BFGS lowers the loss by a further factor of two to twenty and is where much of the
final accuracy of the best models comes from. The test error follows the training loss
throughout; after L-BFGS the gap is within a factor of three for the models with the right
integrals and up to a factor of seven for plain shear networks (Tao's system, charged
particle), which fit the training set better than they generalize. No run in this figure
diverged; divergence during training was observed only for some seeds of the learned-type
core in conjugated models (Section~\ref{sec:conj}), where it is reported.

\begin{figure}[p]
\centering
\includegraphics[width=\textwidth]{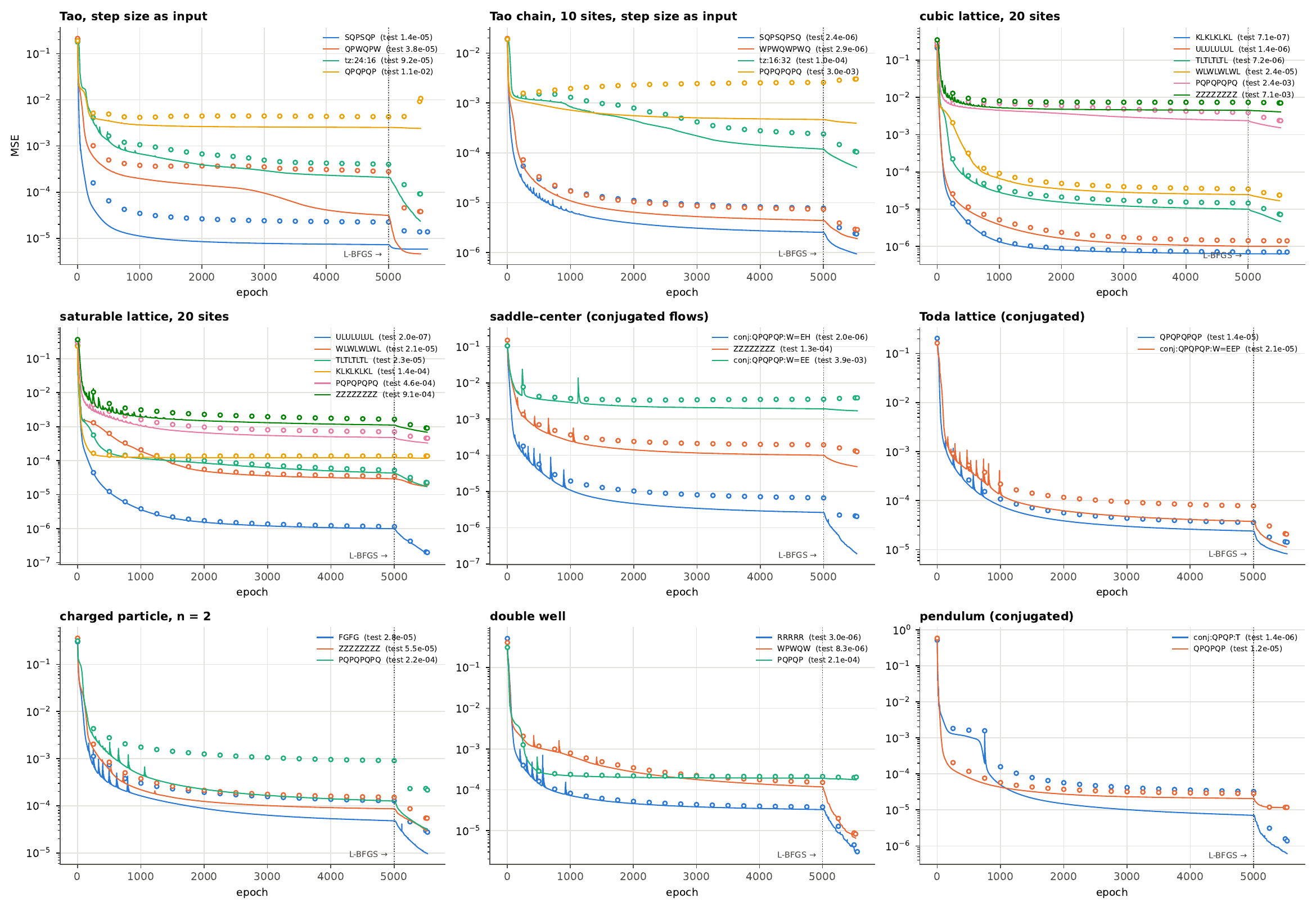}
\caption{Training loss per epoch (lines) and test error every 250 epochs (circles) for one
representative run per model and system. Adam for 5000 epochs, then L-BFGS (dotted line).}
\label{fig:convergence}
\end{figure}

\subsection{Width does not help; depth and layer type do}

Figure~\ref{fig:width} shows the one-step test error when the hidden sizes of all potentials
are multiplied by two and four at fixed depth. For $q$/$p$ shears, learned-plane shears and
fiber layers alike, a twelve-fold increase in parameters leaves the error unchanged. The
error is determined by the number and the kind of layers. Figure~\ref{fig:depth} shows the
complementary experiment: networks made only of learned-type layers reach a given accuracy
with fewer layers than shear networks in one degree of freedom (with three layers,
$1.1\times10^{-6}$ against $4.6\times10^{-4}$ on the pendulum and $5.3\times10^{-5}$ against
$3.8\times10^{-2}$ on Tao's system), and all families converge at eight layers.

\begin{figure}[t]
\centering
\includegraphics[width=\textwidth]{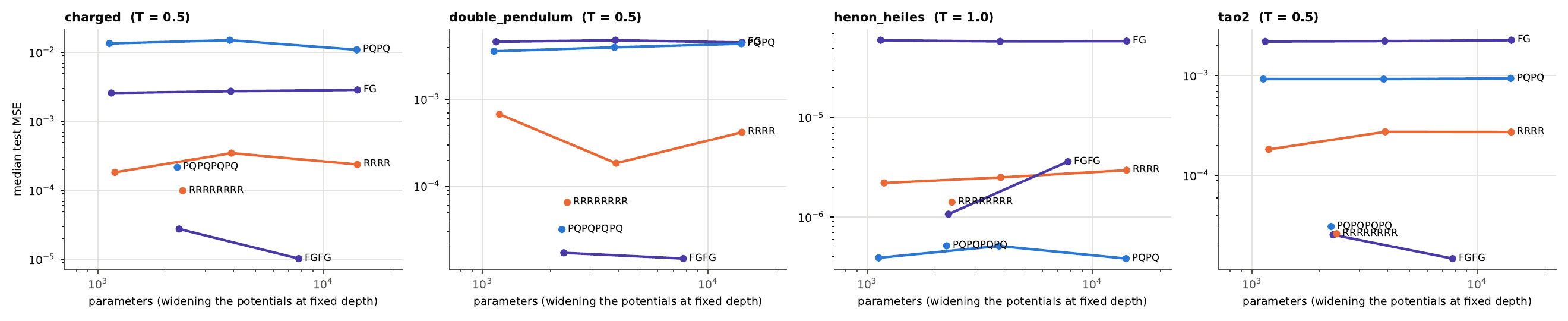}
\caption{Width study at fixed depth, two degrees of freedom. Lines join the same architecture
at hidden sizes $(20,10)$, $(40,20)$, $(80,40)$; isolated points are deeper networks.}
\label{fig:width}
\end{figure}

\begin{figure}[t]
\centering
\includegraphics[width=0.9\textwidth]{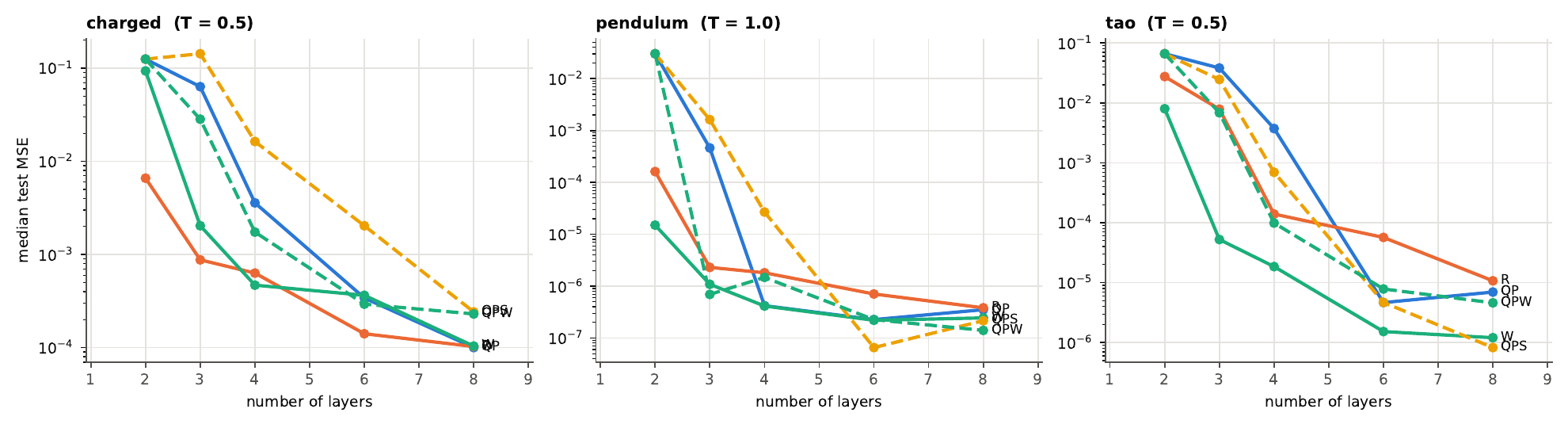}
\caption{Depth study: median test error against the number of layers, by layer family.}
\label{fig:depth}
\end{figure}

\subsection{Step size as an input}

Networks are evaluated at $t=h$ and trained on pairs $(x,h,\varphi_h(x))$, $h\in[0,h_{\max}]$.
Table~\ref{tab:stepped} confirms the obstruction of \cite{janik2025time}: on the systems
with non-separable Hamiltonian in the class $V+K+W(q\odot p)$, $q$/$p$-shear networks have
relative error about $40\%$ on small steps and recover the Hamiltonian only to $22$--$32\%$,
whereas networks containing stretch or learned-type layers recover it to $1$--$2\%$
(Figure~\ref{fig:generators}). Learned-plane shears, whose generators are not restricted to
separable Hamiltonians, lie in between ($4$--$10\%$). In one-step error they are $4$ and $16$
times less accurate than the stretch networks in one and two degrees of freedom, and $70$
and $45$ times less accurate on the chains of six and ten oscillators. On
separable systems all families perform alike. On the two non-separable systems outside the
class, learned-plane shears are best and the stretch layer is the least accurate model: the
benefit of the stretch layer is confined to its generator class. The twist layer, which
carries the integral $q_i^2+p_i^2$, fails on the chains (generator error $20\%$) and wins on
the lattice of Section~\ref{sec:lattice}. There the stretch layer is not excluded, since
$(q^2+p^2)^2=q^4+p^4+2(qp)^2$ also lies in its class, but it is ten times less accurate than
the twist layer ($6.1\cdot10^{-5}$ against $5.6\cdot10^{-6}$ at six sites). The learned-type
layer is within a factor of six of the better specialist in both cases, without being told
which integral the system has.

\begin{table}[t]
\centering\small
\caption{Step size as input: one-step test MSE, median of three seeds (five for the Tao-type
rows); $h_{\max}=0.5$, and $0.3$ for the last four rows. The learned-plane column
reports the better of two configurations of the layer of \cite{tapley2024symplectic}: ours
(two hidden layers, as for the other models) and the one of the author's public code (one
hidden layer, initialization at scale $0.01$, 12--24 layers); on the Tao-type rows the
latter is better. Six layers for $n\le2$, eight for the
chains (periodic chains of Tao oscillators coupled through $q_ip_iq_{i+1}p_{i+1}/2$).
}
\label{tab:stepped}
\begin{tabular}{lccccc}
\toprule
system & \texttt{QPQPQP} & G-SympNet & learned-plane & with S & with W\\
\midrule
pendulum & $2.8\cdot10^{-7}$ & $5.6\cdot10^{-7}$ & $1.0\cdot10^{-7}$ & $1.2\cdot10^{-7}$ & $6.0\cdot10^{-8}$\\
saddle & $5.1\cdot10^{-8}$ & $3.2\cdot10^{-7}$ & $2.6\cdot10^{-7}$ & $1.2\cdot10^{-7}$ & $5.3\cdot10^{-8}$\\
H\'enon--Heiles & $3.1\cdot10^{-7}$ & $9.9\cdot10^{-7}$ & $9.2\cdot10^{-8}$ & $3.3\cdot10^{-7}$ & $2.2\cdot10^{-7}$\\
Tao & $8.0\cdot10^{-3}$ & $2.3\cdot10^{-3}$ & $6.7\cdot10^{-5}$ & $1.6\cdot10^{-5}$ & $8.2\cdot10^{-5}$\\
coupled Tao & $7.0\cdot10^{-4}$ & $7.9\cdot10^{-4}$ & $2.1\cdot10^{-5}$ & $1.3\cdot10^{-6}$ & $1.1\cdot10^{-5}$\\
Tao chain, 6 sites & $4.9\cdot10^{-4}$ & -- & $4.2\cdot10^{-5}$ & $6.0\cdot10^{-7}$ & $1.1\cdot10^{-6}$\\
Tao chain, 10 sites & $3.0\cdot10^{-3}$ & -- & $1.0\cdot10^{-4}$ & $2.2\cdot10^{-6}$ & $2.2\cdot10^{-6}$\\
charged particle & $2.4\cdot10^{-3}$ & $5.5\cdot10^{-4}$ & $5.1\cdot10^{-5}$ & $8.1\cdot10^{-3}$ & $4.0\cdot10^{-4}$\\
double pendulum & $2.5\cdot10^{-4}$ & $5.8\cdot10^{-4}$ & $5.5\cdot10^{-5}$ & $6.8\cdot10^{-4}$ & $3.3\cdot10^{-4}$\\
\bottomrule
\end{tabular}
\end{table}

\begin{figure}[t]
\centering
\includegraphics[width=\textwidth]{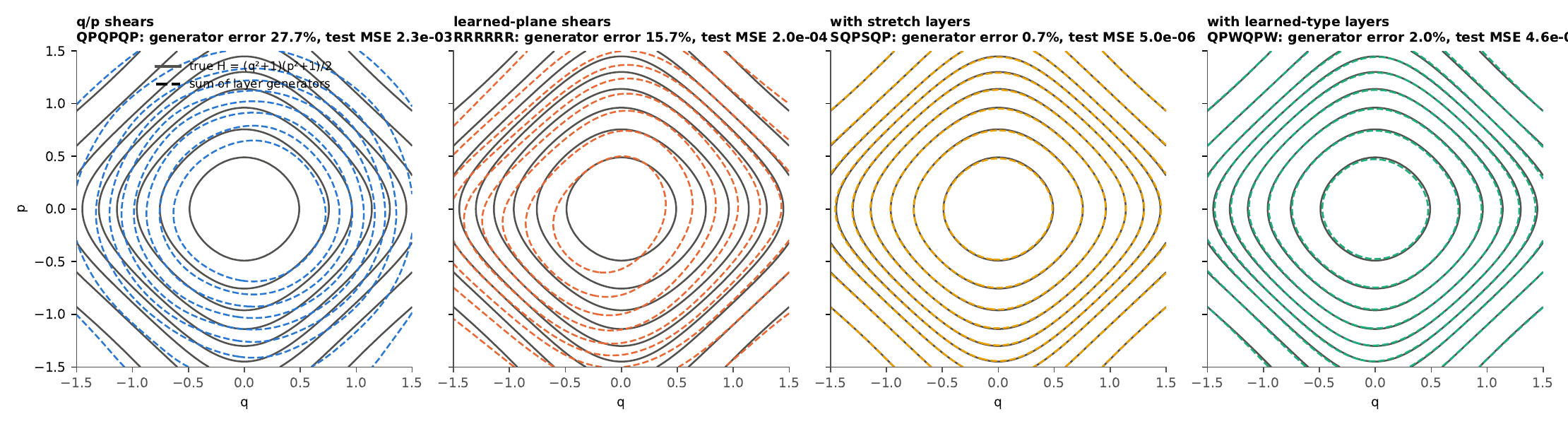}
\caption{Tao's system, step size as input. Level sets of the true Hamiltonian (gray) and of
the sum of the layer generators (dashed), at the same level values.}
\label{fig:generators}
\end{figure}

Figures~\ref{fig:tao-stepped} and~\ref{fig:chain-stepped} show the learned maps iterated
from test initial conditions. On Tao's system the twelve-layer learned-plane network and the
stretch network both follow the true orbits for eighty steps, while the $q$/$p$-shear network
drifts across energy levels; on the ten-site chain only the networks with stretch or
learned-type layers stay on the solution, and the learned-plane network leaves it after
about ten steps.

\begin{figure}[t]
\centering
\includegraphics[width=\textwidth]{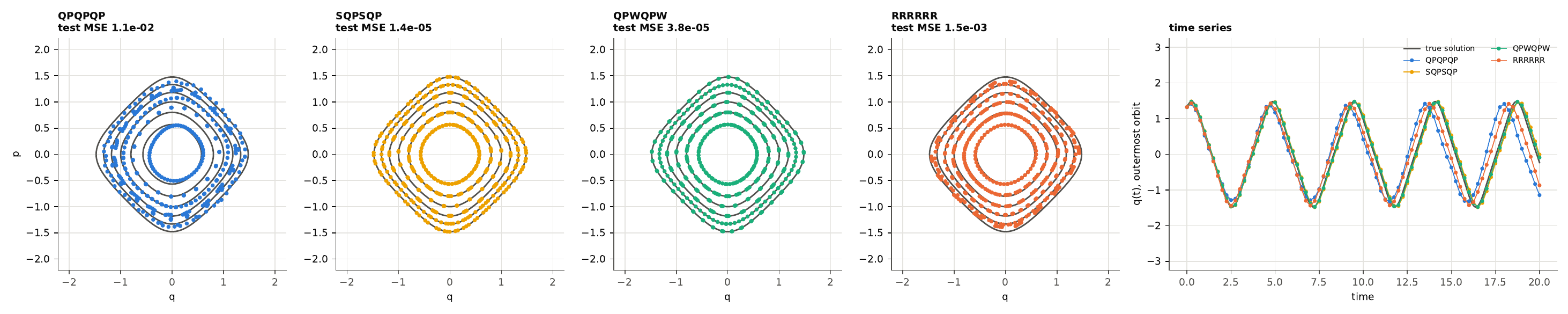}
\caption{Tao's system, step size as input: true orbits (gray) and the learned maps iterated
80 times with step $0.25$ from six test initial conditions; right: $q(t)$ on the outermost
orbit.}
\label{fig:tao-stepped}
\end{figure}

\begin{figure}[t]
\centering
\includegraphics[width=\textwidth]{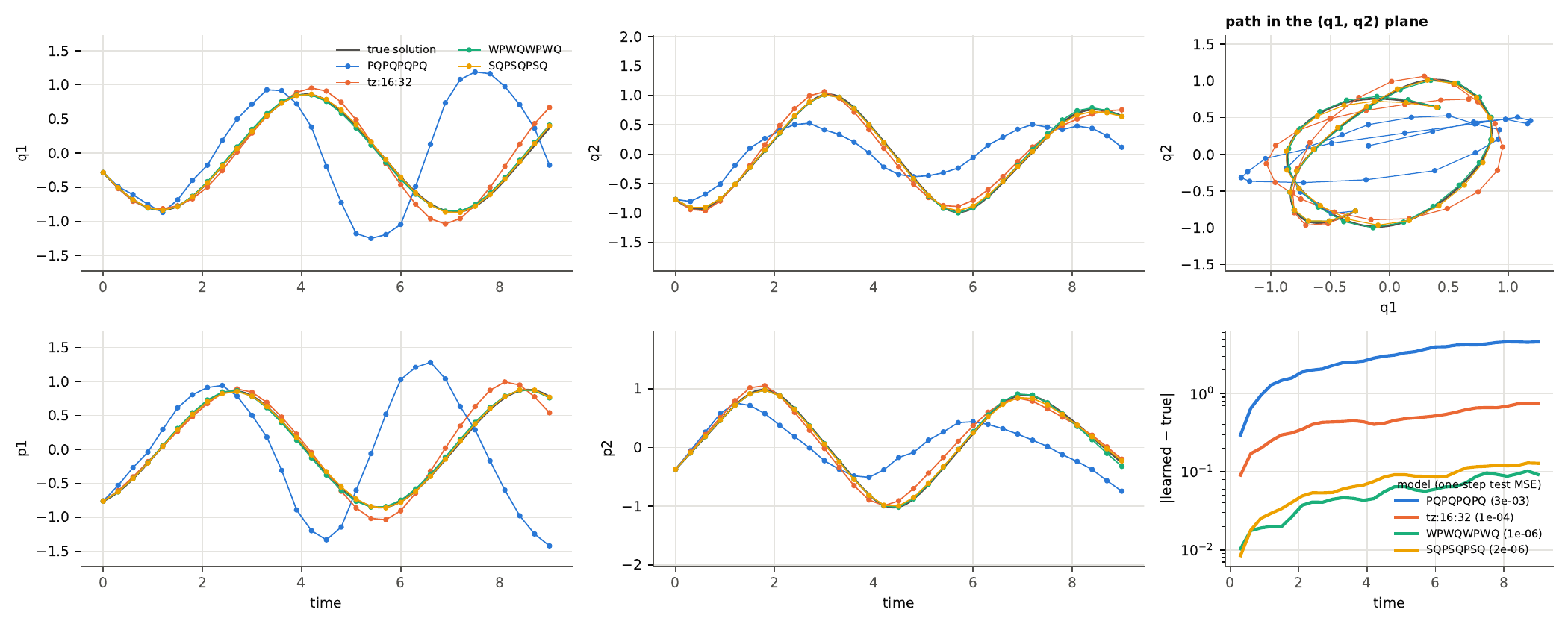}
\caption{Chain of ten Tao oscillators, step size as input: true solution (gray) and the
learned maps iterated 30 times with step $0.3$ from one test initial condition (first two
sites shown); bottom right: error against time.}
\label{fig:chain-stepped}
\end{figure}

\subsection{Conjugated integrable flows}\label{sec:conj}

Model \eqref{eq:conj} is trained on $t\in[0,2]$ and evaluated, by a single evaluation, far
beyond. For orbits that remain in the training region, the error at $t=40$ is
$5\cdot10^{-3}$--$7\cdot10^{-3}$ on the pendulum and $2\cdot10^{-2}$--$5\cdot10^{-2}$ on Tao's
system, against $10^{-1}$ and $>10$ for step-size-aware networks iterated twenty times
(Figure~\ref{fig:longtime}). The learned conserved quantity is a monotone function of the
true energy (rank correlation above $0.9998$).

\begin{figure}[t]
\centering
\includegraphics[width=\textwidth]{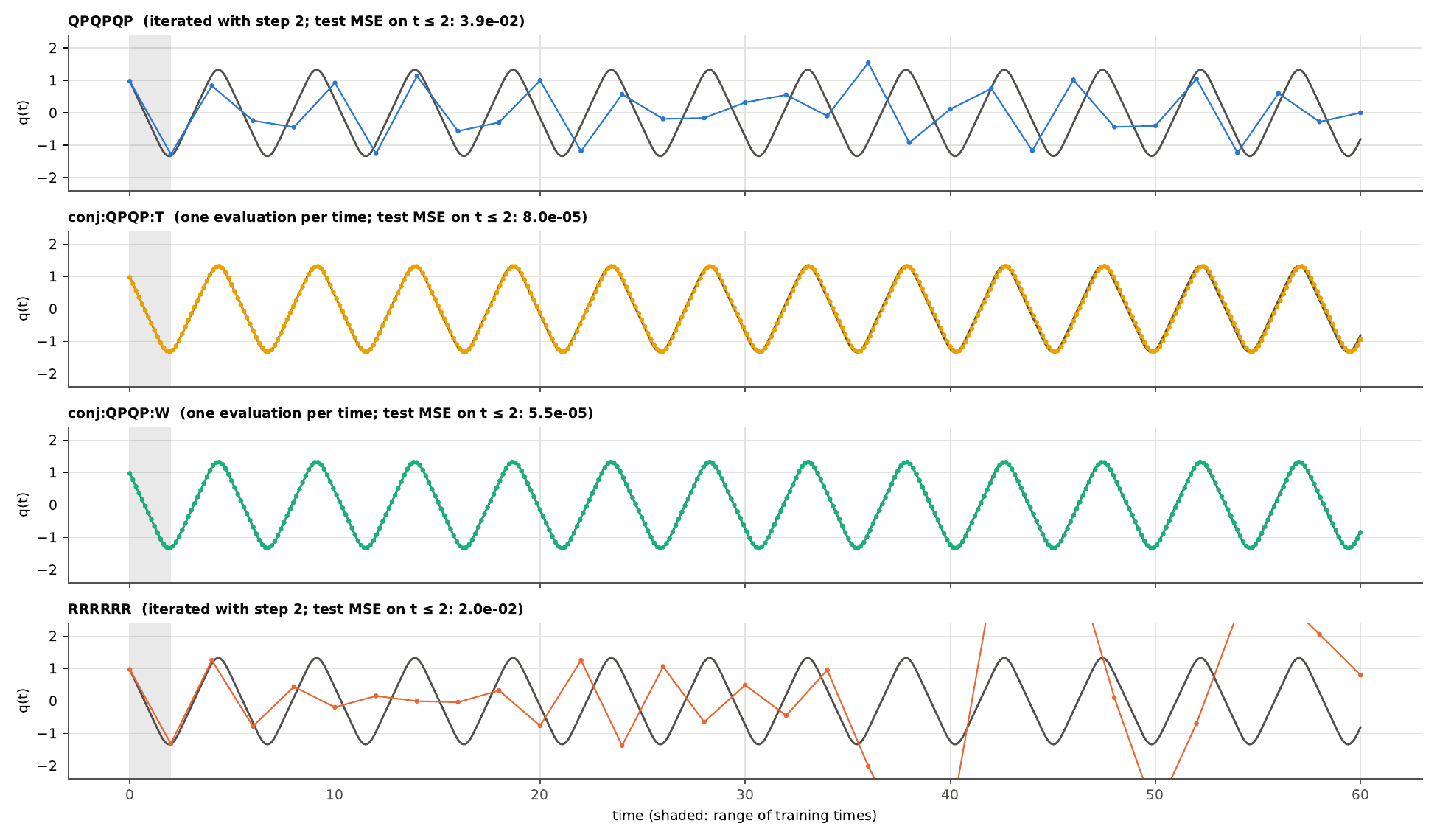}
\caption{Tao's system: true solution (gray) and learned solutions, trained on $t\le2$.}
\label{fig:longtime}
\end{figure}

\begin{figure}[t]
\centering
\includegraphics[width=\textwidth]{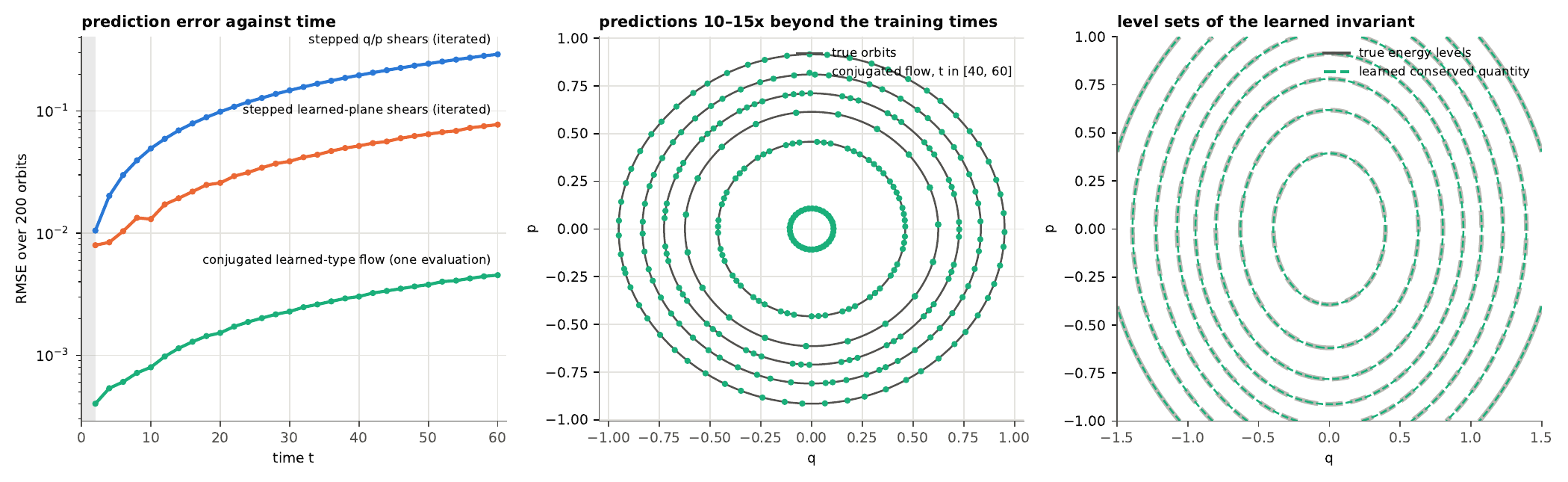}
\caption{Pendulum: conjugated learned-type flow against step-size-aware networks, trained on
$t\le2$. Left: prediction error against time to $t=60$. Middle: predictions at
$t\in[40,60]$ (dots) on the true orbits. Right: level sets of the learned conserved quantity
over the true energy levels.}
\label{fig:pendulum-long}
\end{figure}

\paragraph{Mixed type.}
The saddle--center system has one elliptic and one hyperbolic pair behind a nonlinear change
of variables, so neither a twist core nor a stretch core can represent it. With a
learned-type core the model finds one type of each and reaches test MSE
$2\cdot10^{-6}$--$3\cdot10^{-6}$ in three of five seeds, against $1.3\cdot10^{-4}$ for the best
step-size-aware network, $3.8\cdot10^{-3}$ for a twist core and $7.9\cdot10^{-3}$ for a
stretch core (Figure~\ref{fig:mixed}).

\begin{figure}[t]
\centering
\includegraphics[width=\textwidth]{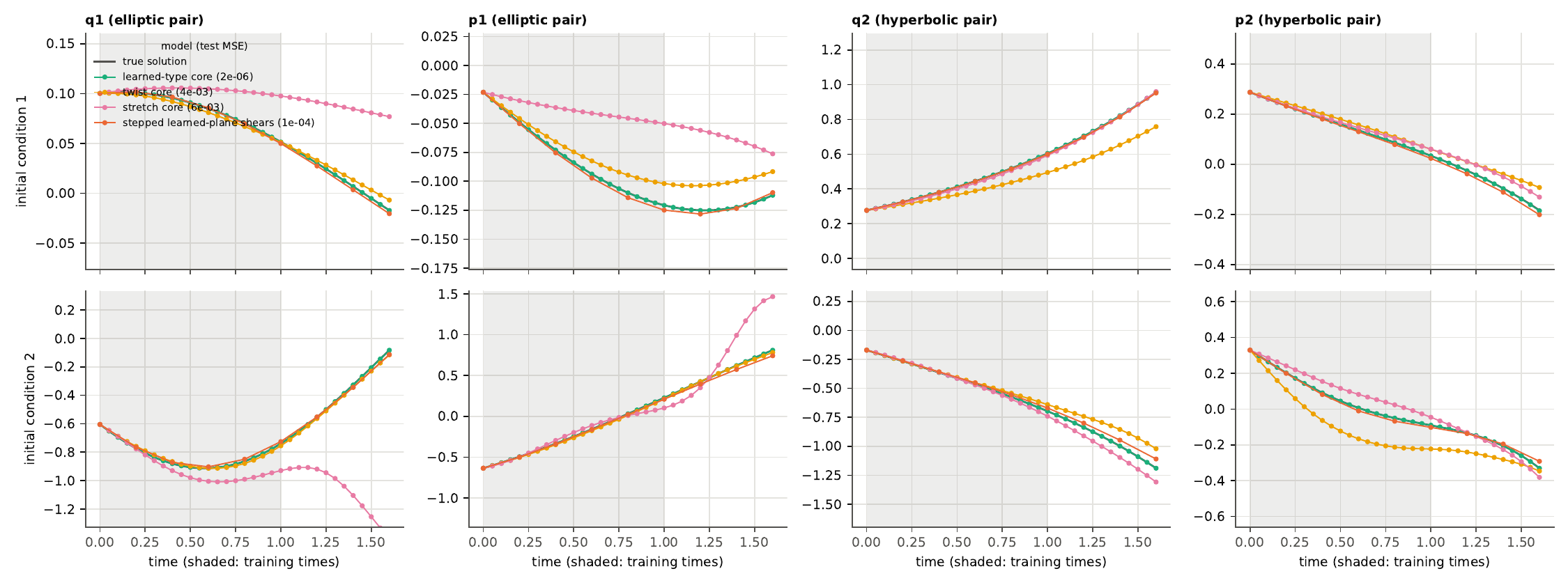}
\caption{Saddle--center system: true solution (gray) and learned solutions.}
\label{fig:mixed}
\end{figure}

\paragraph{Reliability.}
With the correct type fixed in advance, conjugated models train reliably: over ten seeds a
twist core fits the pendulum and Tao's system in ten runs each, and a stretch core fits the
saddle in eight. With a learned-type core a single run succeeds in 4--9 of ten runs,
depending on the system and on the initialization of $M_j$ and $\psi$; runs that end in the
wrong type never fit. Not knowing the type therefore costs reliability. The training loss
identifies the successful runs in every case, so selecting the best of three restarts by
training loss succeeds with probability $0.83$--$1$; when it succeeds, the learned-type core
is as accurate as the fixed one.
For few degrees of freedom there is a more reliable protocol: fix the types in the core,
train one model per type assignment, and keep the one with the lowest training loss. On the
saddle--center system the correct assignment trains in five of five runs (test MSE
$5\cdot10^{-7}$--$6\cdot10^{-6}$), and the three wrong assignments are separated from it by four
orders of magnitude in training loss.

\paragraph{Toda lattice.}
The periodic Toda lattice with three particles is integrable with two oscillatory pairs and
one free pair (the center of mass), so the correct core is elliptic--elliptic--parabolic.
With these types fixed the conjugated model fits in five of five runs; with the free pair
typed as elliptic, in three; with learned types, in one. For initial states with zero total
momentum the error at $t=40$ (training on $t\le2$) is $2.8\cdot10^{-2}$, against
$5.8\cdot10^{-2}$ and $2.1\cdot10^{-1}$ for step-size-aware $q$/$p$-shear and learned-plane
networks iterated twenty times: a gain of two to eight, much smaller than in one degree of
freedom. For states with non-zero total momentum the center of mass leaves the training
region and no model extrapolates.

Figure~\ref{fig:toda} shows the three positions along one zero-momentum orbit: the
conjugated model with the correct types follows all three particles to $t=40$; with a twist
core the phase drifts; the iterated shear networks lose the solution within a few steps.

\begin{figure}[t]
\centering
\includegraphics[width=\textwidth]{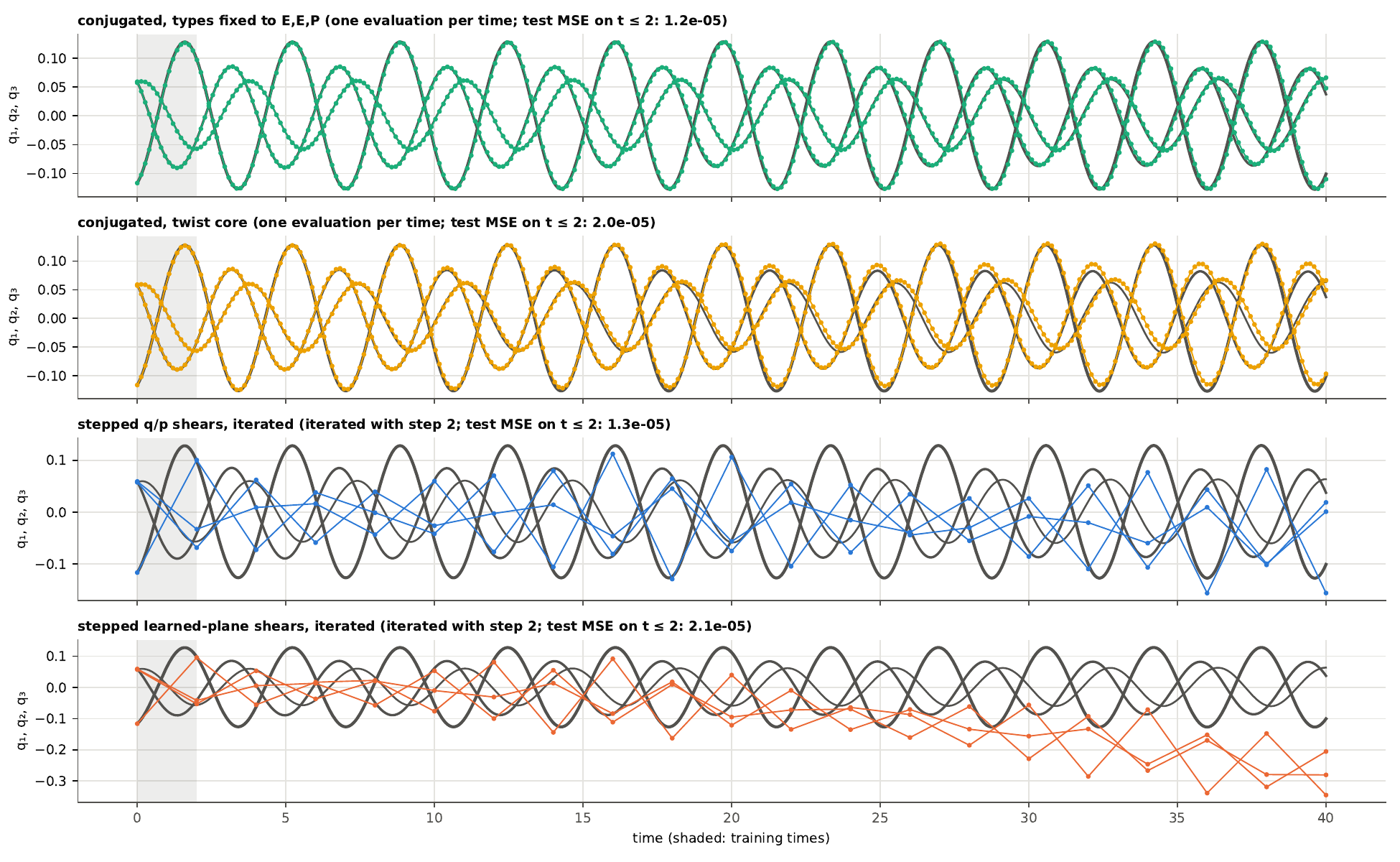}
\caption{Periodic Toda lattice with three particles, zero total momentum: true positions
(gray) and learned solutions, trained on $t\le2$ (shaded).}
\label{fig:toda}
\end{figure}

\subsection{Fixed-time flow maps in one and two degrees of freedom}

Among five-layer models, networks of learned-type layers are the most accurate on the
pendulum, the double well ($4.1\cdot10^{-7}$ against $2.0\cdot10^{-4}$ for $q$/$p$ shears;
Figure~\ref{fig:doublewell}), Tao's system and the coupled Tao system. The stretch layer
alone gives no gain over shears, including on the saddle, where plain shears are the most
accurate at every horizon tried, and on the polar map used in \cite{he2024deep}.
Table~\ref{tab:twodof} compares four fiber layers with eight shear layers at equal parameter
count. The fiber network is the most accurate model on the two non-separable systems outside
the class $V+K+W(q\odot p)$, by a factor of about two over learned-plane shears.

\begin{figure}[t]
\centering
\includegraphics[width=\textwidth]{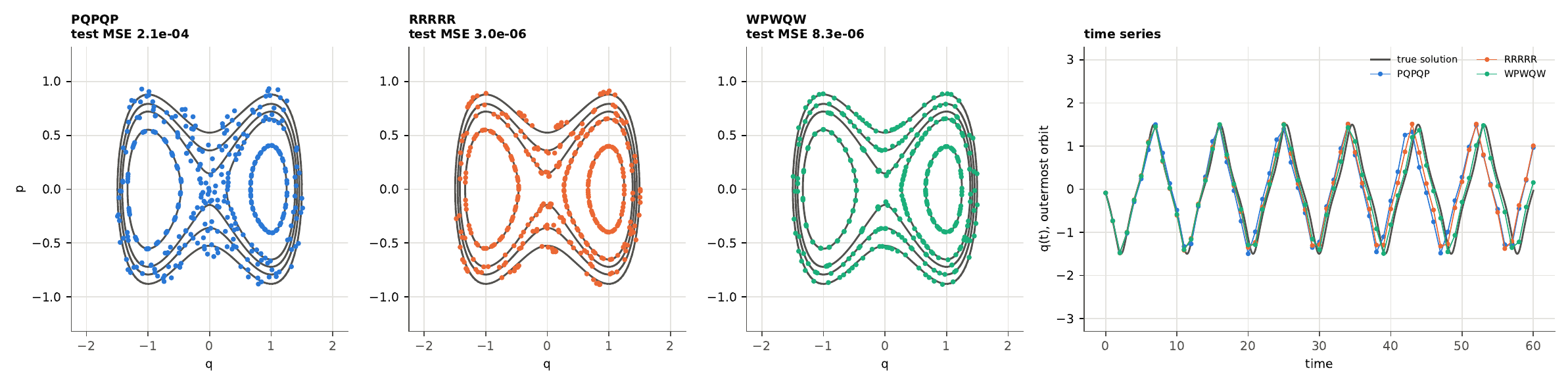}
\caption{Double well: true orbits (gray) and 60 iterates of the learned time-1 map.}
\label{fig:doublewell}
\end{figure}

\paragraph{The polar map.} The map $(q,p)\mapsto(\sqrt{2q}\cos p,\sqrt{2q}\sin p)$ of
\cite{he2024deep} is the transformation from action--angle to Cartesian coordinates. On the
unit square the learned-type networks reach $7\cdot10^{-6}$ to $8\cdot10^{-6}$, plain shears and
the stretch network of the original paper $2.6\cdot10^{-5}$ to $2.7\cdot10^{-5}$. On the annular
domain $[\tfrac12,\tfrac32]\times[0,\tfrac{3\pi}2]$ the cotangent-lift network is best
($2.6\cdot10^{-4}$), and Figure~\ref{fig:polar} shows why: it reproduces the image of a
coordinate grid, whereas the stretch and learned-type networks fold near the two ends of
the annulus. The stretch layer gives no gain over shears on the example it was introduced
with.

\begin{figure}[t]
\centering
\includegraphics[width=\textwidth]{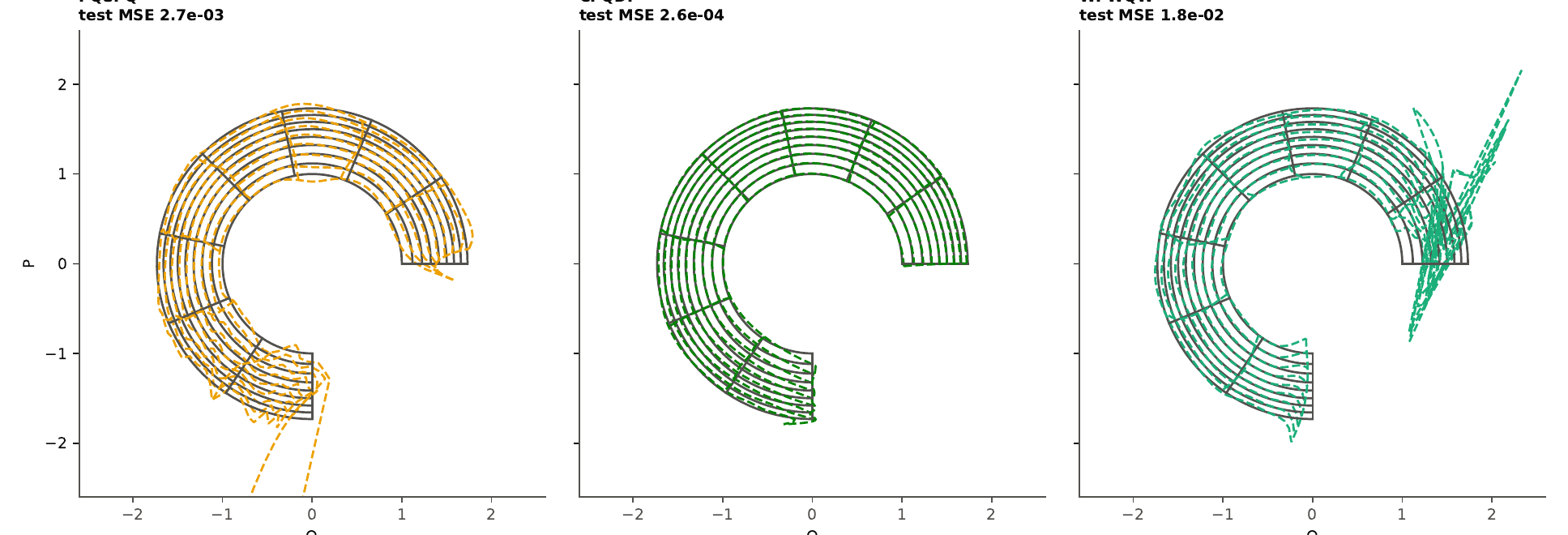}
\caption{Polar map on the annular domain: image of a coordinate grid under the true map
(gray) and under three learned maps (dashed).}
\label{fig:polar}
\end{figure}

\paragraph{Iterated maps in one and two degrees of freedom.} Figures~\ref{fig:pendulum-flow}
and~\ref{fig:charged-flow} iterate the learned time-$T$ maps. On the pendulum all three
model families stay on the true orbits for 60 steps; on the charged particle, a generic
non-separable system, the fiber and learned-plane networks track the solution for 30 steps
with errors a few times smaller than plain shears, in line with the one-step errors of
Table~\ref{tab:twodof}.

\begin{figure}[t]
\centering
\includegraphics[width=\textwidth]{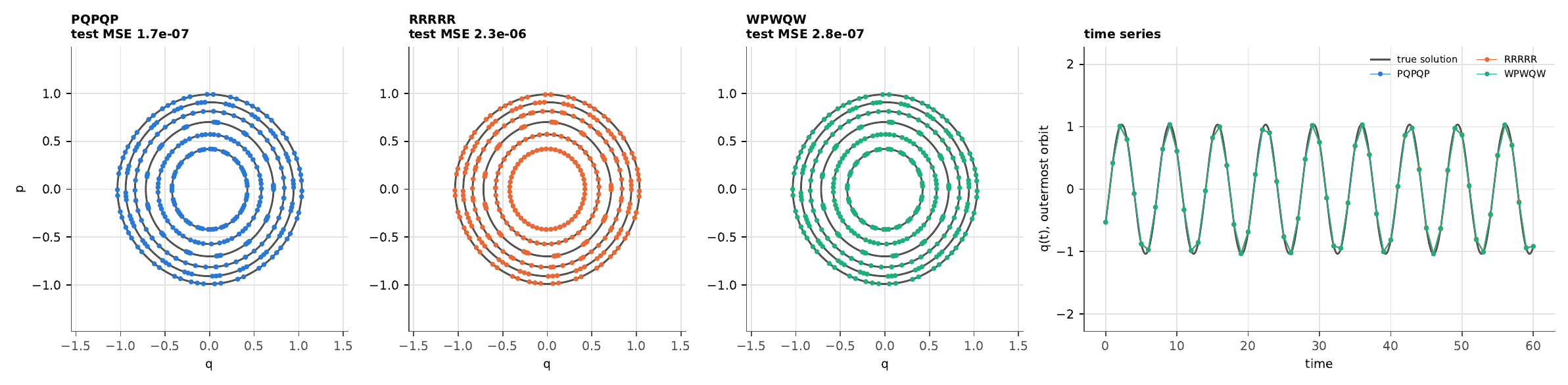}
\caption{Pendulum: true orbits (gray) and 60 iterates of three learned time-1 maps.}
\label{fig:pendulum-flow}
\end{figure}

\begin{figure}[t]
\centering
\includegraphics[width=\textwidth]{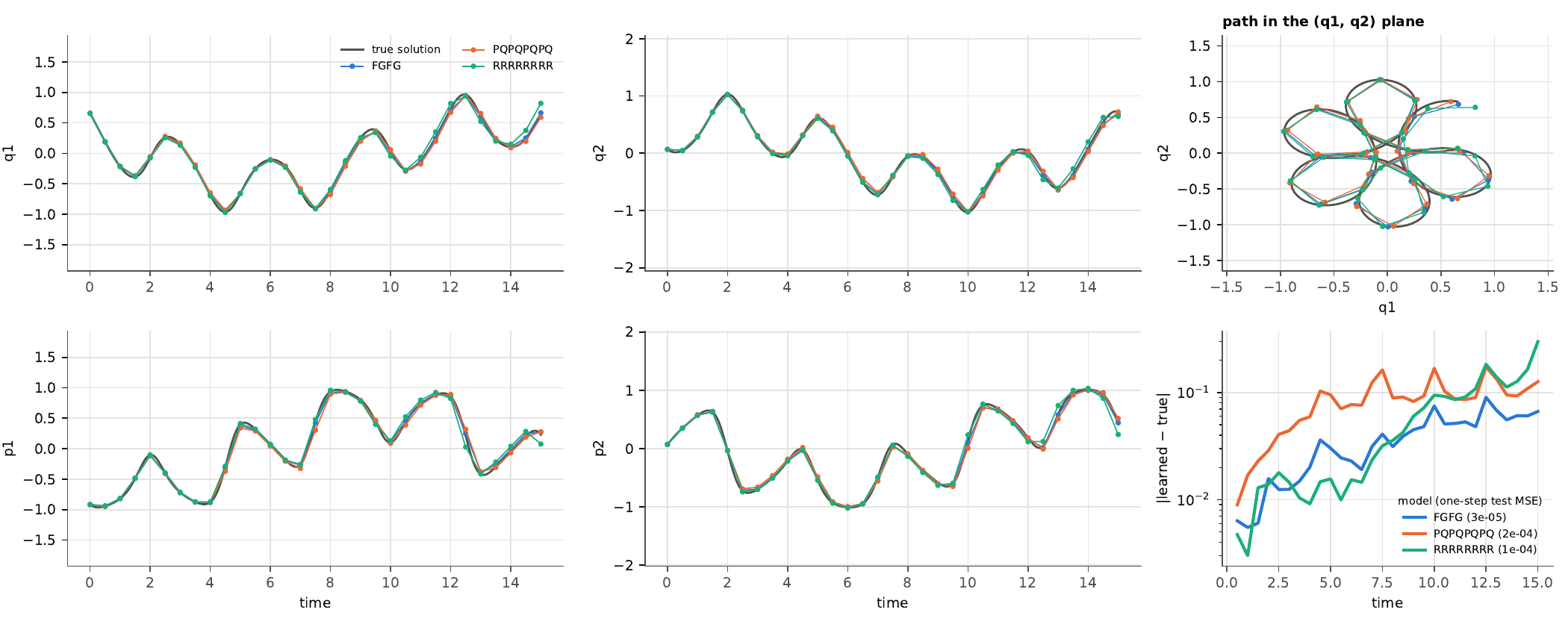}
\caption{Charged particle in a non-uniform magnetic field ($n=2$): true solution (gray) and
30 iterates of the learned time-$0.5$ maps.}
\label{fig:charged-flow}
\end{figure}

\begin{table}[t]
\centering\small
\caption{Two degrees of freedom, fixed-time maps: one-step test MSE, median (range) over five
seeds, about 2300 parameters per model.}
\label{tab:twodof}
\begin{tabular}{lccc}
\toprule
system & \texttt{FGFG} & $q$/$p$ shears $\times8$ & learned-plane $\times8$ \cite{tapley2024symplectic}\\
\midrule
charged particle & $2.0\cdot10^{-5}$ $(0.9$--$17)$ & $2.5\cdot10^{-4}$ & $4.0\cdot10^{-5}$ $(2.8$--$6.9)$\\
double pendulum & $1.7\cdot10^{-5}$ $(1.6$--$1.9)$ & $3.5\cdot10^{-5}$ & $3.0\cdot10^{-5}$ $(2.3$--$12)$\\
coupled Tao & $1.9\cdot10^{-5}$ & $3.6\cdot10^{-5}$ & $1.6\cdot10^{-5}$\\
H\'enon--Heiles & $5.6\cdot10^{-7}$ & $4.7\cdot10^{-7}$ & $8.6\cdot10^{-7}$\\
\bottomrule
\end{tabular}
\end{table}

\subsection{Nonlinear Schr\"odinger lattices}\label{sec:lattice}

The periodic lattice
\[
 H=\sum_i\tfrac12\bigl[(q_{i+1}-q_i)^2+(p_{i+1}-p_i)^2\bigr]+g(q_i^2+p_i^2)
\]
has a linear coupling and an on-site nonlinearity that is a function of the integrals
$q_i^2+p_i^2$ carried by twist and learned-type layers. We learn its time-$0.3$ map from 2000
samples for the cubic case $g(s)=s^2/4$ with $n=6$, $10$, $20$, $40$ sites and for the
saturable case $g(s)=\log(1+s)$ with $n=10$, $20$. Besides alternating the new layers with
shears, we alternate them with linear symplectic maps L. We compare four ways of
parametrizing the potential $\Phi$ of the twist layers:
\begin{itemize}
 \item \texttt{TLTLTLTL}: a general dense network $\Phi(I_1,\dots,I_n)$;
 \item \texttt{ULULULUL}: an on-site potential of learned form, $\Phi=\sum_i\phi(I_i)$ with one
 small network $\phi$ shared by all sites;
 \item \texttt{KLKLKLKL}: the fixed cubic (Kerr) form $\Phi=\sum_i(aI_i+\gamma I_i^2)$ with two
 trained scalars per layer. This is the learned split-step scheme of optical communications
 \cite{hager2018nonlinear,hager2021physics}, where the nonlinear step is fixed and the
 linear filters are trained; here both are trained;
 \item \texttt{WLWLWLWL}: learned-type layers, which are not told the integral.
\end{itemize}
All four are exactly symplectic instances of Lemma~\ref{lem:commuting}.

\begin{table}[t]
\centering\small\setlength{\tabcolsep}{4pt}
\caption{Lattices: one-step test MSE of the time-$0.3$ map, median over five seeds (three for
\texttt{ULULULUL} at 6 and 40 sites and on the saturable lattice, and for the last three
rows). ``Learned-plane'' is the better of our configuration and the author's.
Bottom: Strang splitting with the true Hamiltonian, for orientation.}
\label{tab:lattice}
\begin{tabular}{lcccccc}
\toprule
 & \multicolumn{4}{c}{cubic} & \multicolumn{2}{c}{saturable}\\
model & 6 & 10 & 20 & 40 & 10 & 20\\
\midrule
\texttt{ULULULUL} (site-wise twist) & $\mathbf{8.6\cdot10^{-8}}$ & $\mathbf{1.3\cdot10^{-7}}$ & $\mathbf{1.4\cdot10^{-7}}$ & $\mathbf{2.4\cdot10^{-7}}$ & $\mathbf{1.3\cdot10^{-7}}$ & $\mathbf{2.0\cdot10^{-7}}$\\
\texttt{KLKLKLKL} (fixed cubic form) & $4.3\cdot10^{-7}$ & $5.0\cdot10^{-7}$ & $7.2\cdot10^{-7}$ & $1.2\cdot10^{-6}$ & $1.3\cdot10^{-4}$ & $1.4\cdot10^{-4}$\\
\texttt{TLTLTLTL} (general twist) & $1.1\cdot10^{-6}$ & $2.9\cdot10^{-6}$ & $7.0\cdot10^{-6}$ & $3.5\cdot10^{-5}$ & $6.6\cdot10^{-6}$ & $2.1\cdot10^{-5}$\\
\texttt{TPTQTPTQ} (twist + shears) & $5.6\cdot10^{-6}$ & $1.3\cdot10^{-5}$ & $7.6\cdot10^{-5}$ & $2.6\cdot10^{-4}$ & -- & --\\
\texttt{WLWLWLWL} (learned-type) & $2.2\cdot10^{-5}$ & $2.5\cdot10^{-5}$ & $2.4\cdot10^{-5}$ & $3.2\cdot10^{-5}$ & $1.8\cdot10^{-5}$ & $2.1\cdot10^{-5}$\\
\midrule
$q$/$p$ shears $\times8$ & $1.4\cdot10^{-4}$ & $4.0\cdot10^{-4}$ & $2.4\cdot10^{-3}$ & $7.3\cdot10^{-3}$ & $2.2\cdot10^{-4}$ & $4.6\cdot10^{-4}$\\
learned-plane shears \cite{tapley2024symplectic} & $1.3\cdot10^{-4}$ & $6.1\cdot10^{-4}$ & $6.5\cdot10^{-3}$ & $8.6\cdot10^{-3}$ & $9.5\cdot10^{-5}$ & $5.5\cdot10^{-4}$\\
\texttt{LPLQLPLQ} (shears + linear) & $1.8\cdot10^{-4}$ & $4.8\cdot10^{-4}$ & $9.7\cdot10^{-4}$ & -- & -- & --\\
LA-SympNet \cite{jin2020sympnets} & $9.5\cdot10^{-4}$ & $2.2\cdot10^{-3}$ & $3.5\cdot10^{-3}$ & -- & -- & --\\
\midrule
Strang, true $H$, 1 / 2 / 4 stages & \multicolumn{4}{c}{$4.5\cdot10^{-5}$ / $2.4\cdot10^{-6}$ / $1.4\cdot10^{-7}$} & \multicolumn{2}{c}{--}\\
\bottomrule
\end{tabular}
\end{table}

\begin{figure}[t]
\centering
\includegraphics[width=0.75\textwidth]{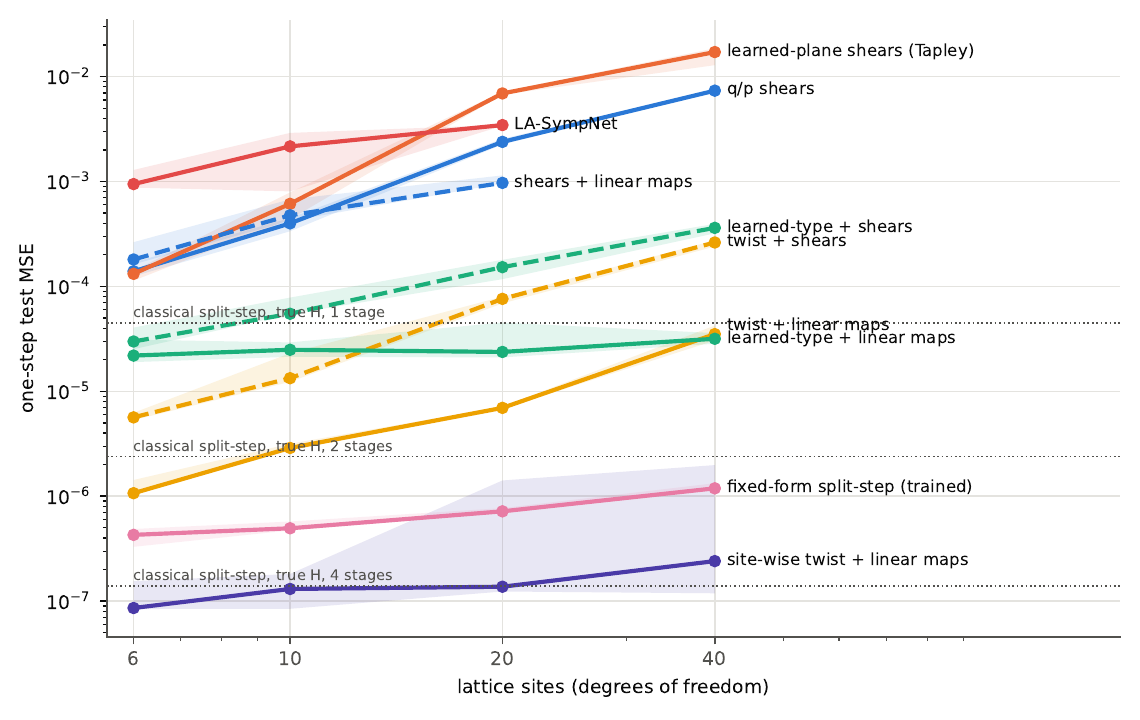}
\caption{Cubic lattice: test error against the number of sites (median and range over
seeds). Dotted: Strang splitting with the true Hamiltonian.}
\label{fig:lattice}
\end{figure}

Table~\ref{tab:lattice} and Figure~\ref{fig:lattice} show the result.

\emph{Shear-based networks.} $q$/$p$ shears, learned-plane shears in either configuration,
shears alternating with linear maps and LA-SympNets all have errors between $10^{-4}$ and
$10^{-2}$, growing with the number of sites. They are less accurate than a single classical
splitting step with the true Hamiltonian. Linear layers alone do not help.

\emph{Twist layers.} Every network with twist or learned-type layers is more accurate than
every shear-based network on every lattice. The general twist network is about $120$,
$140$, $140$ and $210$ times more accurate than the best shear-based model at 6, 10, 20 and
40 sites of the cubic lattice, and $14$ and $22$ times on the saturable lattice. Seed
ranges do not overlap.

\emph{How the potential is parametrized matters most.} With the fixed cubic form, the
network is $3$ to $30$ times more accurate than the general twist network on the cubic
lattice, where its assumption is exactly right, and no better than shear networks on the
saturable lattice, where it is wrong. The site-wise network, which assumes only that the
nonlinearity is on-site and the same at every site, is the most accurate model on both
lattices: $4$ to $5$ times better than the fixed cubic form on the cubic lattice, three
orders of magnitude better on the saturable one, and between $7\cdot10^2$ and $3\cdot10^4$
times better than the best shear-based network. Its error is nearly independent of the number of sites
and comparable to four classical splitting stages with the true Hamiltonian. The comparison
is not symmetric in prior knowledge: the site-wise and fixed-form networks are told that the
nonlinearity is on-site, the general twist network only that it is a function of the
actions, and the shear networks nothing.

With four times the training budget on the cubic 10-site lattice the ranking is unchanged:
$2.4\cdot10^{-8}$ for the fixed cubic form, $4.0\cdot10^{-8}$ for the site-wise network,
$7.9\cdot10^{-7}$ for the general twist network, $6.9\cdot10^{-5}$ for the best learned-plane
configuration, and $2.8\cdot10^{-4}$ and $7.9\cdot10^{-4}$ for G- and LA-SympNets at the
sizes of the three-body example of \cite{jin2020sympnets}. With the longer budget the fixed
cubic form overtakes the site-wise network on the cubic lattice, by a factor of $1.6$.
With the step size as an input, the sum of the layer generators of \texttt{TPTQTPTQ}
reproduces the cubic lattice Hamiltonian to $0.6\%$ ($1.9\%$ with learned-type layers),
against $6$--$7\%$ for the shear networks. Figure~\ref{fig:dnls} shows iterated predictions.

The gain is also one of sample efficiency. At ten sites, with the training set reduced
from 2000 to 250 points, the general twist network reaches $8.2\cdot10^{-6}$ and the
fixed-form network $6.6\cdot10^{-7}$, against $1.1\cdot10^{-2}$ and $3.2\cdot10^{-2}$ for
$q$/$p$ shears and learned-plane shears; with 4000 points the errors are $2.5\cdot10^{-6}$,
$4.9\cdot10^{-7}$, $3.6\cdot10^{-4}$ and $6.4\cdot10^{-4}$ (Figure~\ref{fig:samples}). For all
models the error is flat beyond 1000 points, so the remaining error is not due to lack of
data.

\begin{figure}[t]
\centering
\includegraphics[width=0.6\textwidth]{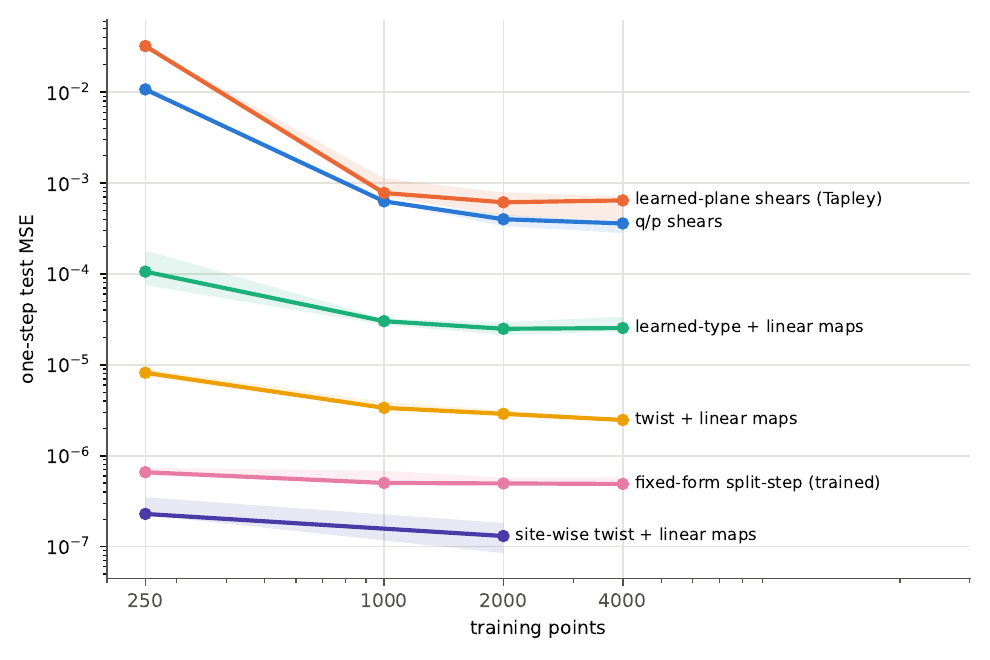}
\caption{Cubic ten-site lattice: test error against the number of training points (median
and range over three to five seeds).}
\label{fig:samples}
\end{figure}

\begin{figure}[t]
\centering
\includegraphics[width=\textwidth]{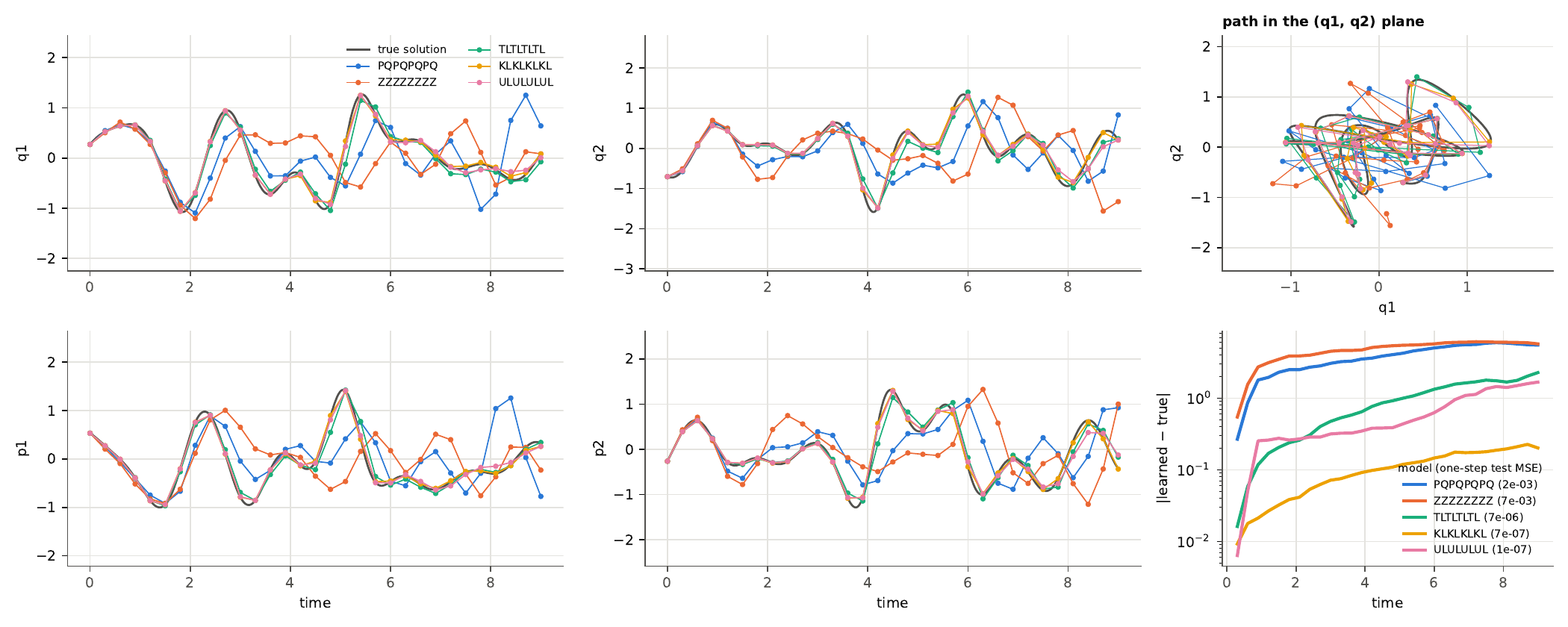}
\caption{Cubic lattice with 20 sites: true solution (gray) and thirty iterates of each
learned time-$0.3$ map from one test initial condition (first two sites shown).}
\label{fig:dnls}
\end{figure}

Figures~\ref{fig:sat20} and~\ref{fig:dnls40} show iterated predictions on the saturable
lattice with 20 sites and the cubic lattice with 40 sites. The ordering of the one-step
errors carries over to the iterated solutions: on the saturable lattice the fixed-form
network, whose nonlinearity is wrong, is no better than the shear networks, while the
site-wise and general twist networks follow the true solution; at 40 sites the fixed-form
and site-wise networks track the solution for thirty steps and the shear networks do not.

\begin{figure}[t]
\centering
\includegraphics[width=\textwidth]{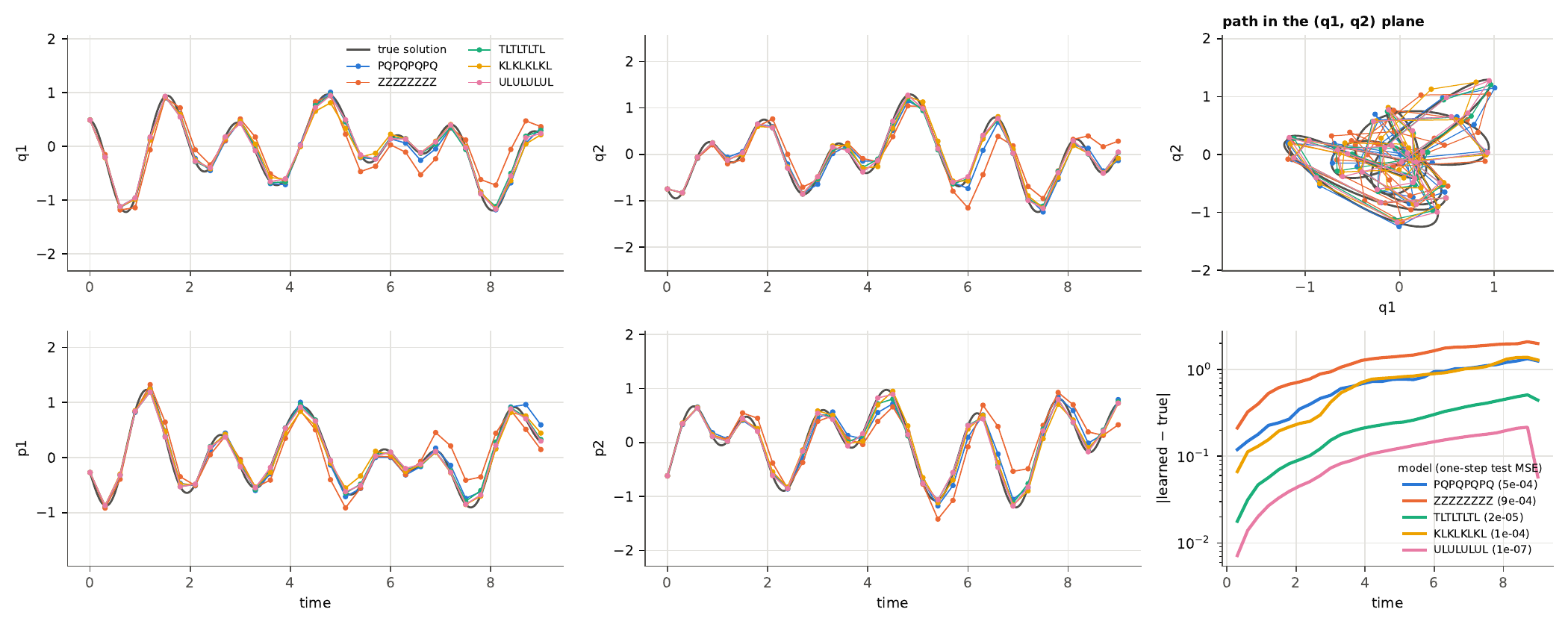}
\caption{Saturable lattice with 20 sites: true solution (gray) and thirty iterates of each
learned time-$0.3$ map from one test initial condition.}
\label{fig:sat20}
\end{figure}

\begin{figure}[t]
\centering
\includegraphics[width=\textwidth]{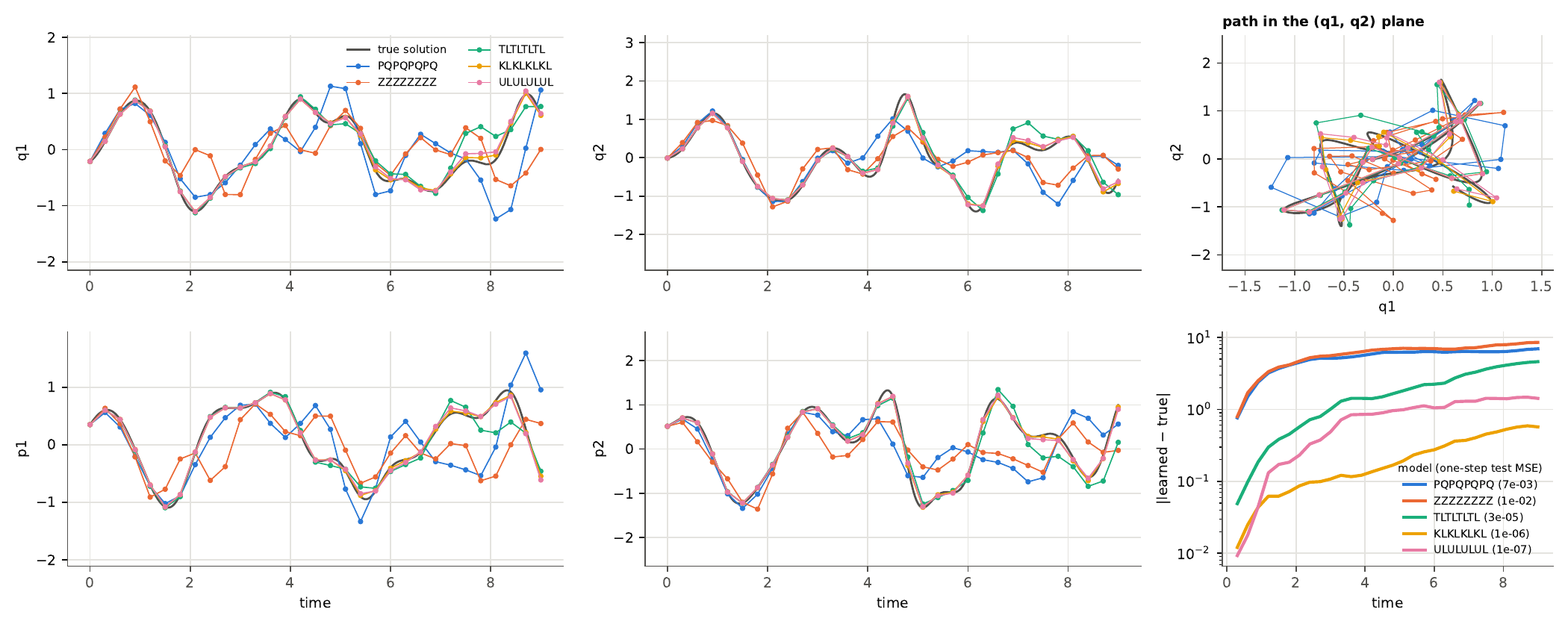}
\caption{Cubic lattice with 40 sites: true solution (gray) and thirty iterates of each
learned time-$0.3$ map from one test initial condition.}
\label{fig:dnls40}
\end{figure}

The lattice conserves the energy and the mass $N=\tfrac12\sum_i(q_i^2+p_i^2)$. None of the
learned maps conserves them exactly, but over 300 iterations at 20 sites
(Figure~\ref{fig:rollout}; medians over five seeds) the site-wise, fixed-form and general
twist networks with linear maps keep both within $0.2$--$0.4\%$ in every seed, and the
twist network with shears within $8\%$ and $5\%$. Learned-plane shears drift by $78\%$ and
$63\%$, and iterates of the $q$/$p$-shear network diverge in every seed. The learned-type
network with linear maps, although accurate over one step, is not reliable over long
rollouts: its energy error ranges from $5\%$ to a factor of several hundred across seeds.
Rollouts leave the box on which the maps were trained; there the fixed-form network, whose
nonlinearity is correct everywhere, has the smallest trajectory error.

\begin{figure}[t]
\centering
\includegraphics[width=\textwidth]{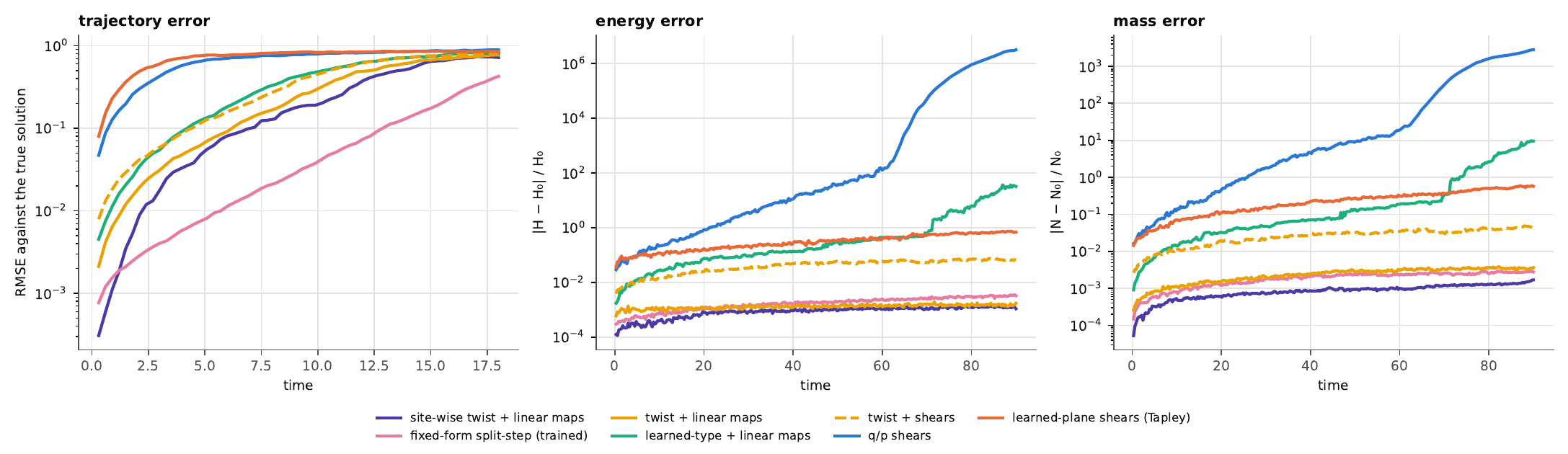}
\caption{Cubic lattice with 20 sites: learned time-$0.3$ maps iterated 300 times; medians
over 50 orbits and five seeds of the error against the true solution and of the relative
errors in energy and mass.}
\label{fig:rollout}
\end{figure}

\subsection{Generic non-separable systems in higher dimension}

On systems without such structure the picture is different (Table~\ref{tab:highdim}). On the
three-link pendulum and on two charged particles the fiber network is the most accurate
model, but only by a factor of about $1.4$ over learned-plane shears; on the separable chain
plain shears are best. With four times the training budget all errors fall by a factor of
three to five and the ranking is unchanged (three-link pendulum: $3.0\cdot10^{-5}$ for the
fiber network against $5.1\cdot10^{-5}$ for learned-plane shears; two charged particles:
$2.1\cdot10^{-5}$ against $3.3\cdot10^{-5}$; two seeds). G-SympNets are one to two orders of
magnitude less accurate than all of these.

\begin{table}[t]
\centering\small
\caption{Generic systems in higher dimension: one-step test MSE, median over five seeds
(three for the FPUT chain).}
\label{tab:highdim}
\begin{tabular}{lcccc}
\toprule
system ($n$) & \texttt{FGFG} & $q$/$p$ shears $\times8$ & learned-plane $\times8$ & \texttt{WPWQWPWQ}\\
\midrule
three-link pendulum (3) & $\mathbf{9.6\cdot10^{-5}}$ & $7.2\cdot10^{-4}$ & $1.4\cdot10^{-4}$ & $8.4\cdot10^{-4}$\\
two charged particles (4) & $\mathbf{1.0\cdot10^{-4}}$ & $2.9\cdot10^{-4}$ & $1.4\cdot10^{-4}$ & $1.6\cdot10^{-4}$\\
FPUT chain (8) & $6.1\cdot10^{-5}$ & $\mathbf{3.5\cdot10^{-5}}$ & $9.8\cdot10^{-5}$ & --\\
\bottomrule
\end{tabular}
\end{table}

\section{Discussion and limitations}

The experiments give a consistent picture. Layers built on \emph{nonlinear} integrals pay off
when the Hamiltonian contains functions of those integrals: products $q_ip_i$ (Tao-type
systems), actions $q_i^2+p_i^2$ (the nonlinear Schr\"odinger lattice), or, after a change of
variables, any integrable system (conjugated flows). There the gains over shear-based
networks are one to four orders of magnitude and persist up to forty degrees of freedom.
Within this class, how much is assumed about the potential decides the accuracy: a general
potential of the integrals, an on-site potential of learned form, and a fixed functional
form are successively more accurate when their assumption holds, and a wrong fixed form is
no better than shears. The learned-type layer, which is not told the integral, is the least
accurate of the new layers on each structured system but the only one that works on all of
them. On generic non-separable systems the picture is different: learned-plane shears
\cite{tapley2024symplectic} are a strong baseline, fiber layers improve on them by a factor
of at most two, and the stretch layer can hurt.

\paragraph{Limitations.}
The comparisons with shear networks on the lattices and chains are not equal
in prior knowledge; they show what the right integral buys, not that the new layers are
better in general. Conjugated flows require integrability and, with a learned-type core,
several restarts or an enumeration of type assignments; beyond one degree of freedom their
long-time advantage was modest. The learned-type network is unstable over long rollouts on
the lattice. The fiber layer has no explicit inverse and nothing constrains $Df$ to stay
invertible, although its smallest singular value stayed above $0.4$ in all our runs.
Widening any of these layers does not help. Most experiments use one training budget, with
the ranking checked at four times that budget on three systems; the largest system has forty
degrees of freedom; and the baselines are re-implementations checked against, but not run
from, the authors' code.

\appendix
\section{Implementation details}\label{app:details}

\paragraph{Potentials.} Unless stated otherwise every scalar potential ($F$, $G$, $\Phi$) is a
dense network with $\tanh$ activation and hidden sizes $(20,10)$ for $n\le2$ and $(40,20)$
for $n\ge3$, with the default initialization of the framework. Gradients of potentials are
taken by automatic differentiation, so layers are symplectic to rounding error for any
smooth potential. The site-wise potential is $\Phi(I)=\sum_i\phi(I_i)$ with $\phi$ a dense
network $1\to(40,20)\to1$; the fixed-form potential is $\sum_i(aI_i+\gamma I_i^2)$ with
$a=0$, $\gamma=0.05$ initially.

\paragraph{Layers.} In the learned-type layer, $M_j$ is initialized with independent standard
normal entries and $c_j$ with standard deviation $0.5$; the functions $C$, $S$, $R$ of
\eqref{eq:quad-layer} are evaluated by their Taylor polynomials of degree three for
$|u|<10^{-3}$. A linear map L is a product of four unit triangular symplectic matrices with
symmetric off-diagonal blocks, initialized at scale $0.01$, followed by a translation. In
the moving-plane layer, $S(a)$ is the symmetrized output of a dense network with one hidden
layer, and the roles of the two groups of coordinates are drawn at random per layer. In the
fiber layers the residual network $B$ has the same hidden sizes as the potentials and its
last layer is scaled by $0.1$ at initialization.

\paragraph{Baselines.} G- and LA-SympNets follow \cite{jin2020sympnets}: sigmoid activation,
parameters initialized at scale $0.01$, biases at zero. H\'enonNet layers are four
compositions of one H\'enon map with a one-hidden-layer $\tanh$ potential
\cite{burby2021henonnets}. Learned-plane shears are run in two configurations: ours
(the potentials above, plane matrices initialized at scale $0.3$, eight layers) and that of
the author's code \cite{tapley2024symplectic} (one hidden layer of width 16--64, all
parameters at scale $0.01$, 8--24 layers); tables report the better one. All baselines were
re-implemented from the papers and, for \cite{jin2020sympnets,tapley2024symplectic}, checked
against the public source code.

\paragraph{Training.} Full-batch Adam for 5000 epochs with learning rate decaying
exponentially from $10^{-2}$ to $10^{-4}$ and gradient-norm clipping at 10, followed by
L-BFGS (strong Wolfe line search, history 50) for at most 500 iterations; the ``long''
budget multiplies both by four. Mean squared error on the outputs, double precision.
Reference solutions use an explicit Runge--Kutta method of order 8 with tolerances
$10^{-12}$. Initial conditions are uniform on a box ($[-1,1]^{2n}$ for the lattices and
chains). Conjugated models use $\psi$ with four or six shear layers.

\paragraph{Long-time metrics.} Where a system has a compact sublevel set of the energy inside
the sampling box, long rollouts and extrapolation in time are evaluated on initial
conditions in that set, so that orbits remain in the region covered by the training data.

\section{Additional results}\label{app:results}

This appendix collects results referred to in the text and the per-seed data behind the
headline tables. Every dot in the per-seed figures is one trained network; bars are medians.

\subsection{Per-seed data for the headline tables}

\begin{figure}[p]
\centering
\includegraphics[width=\textwidth]{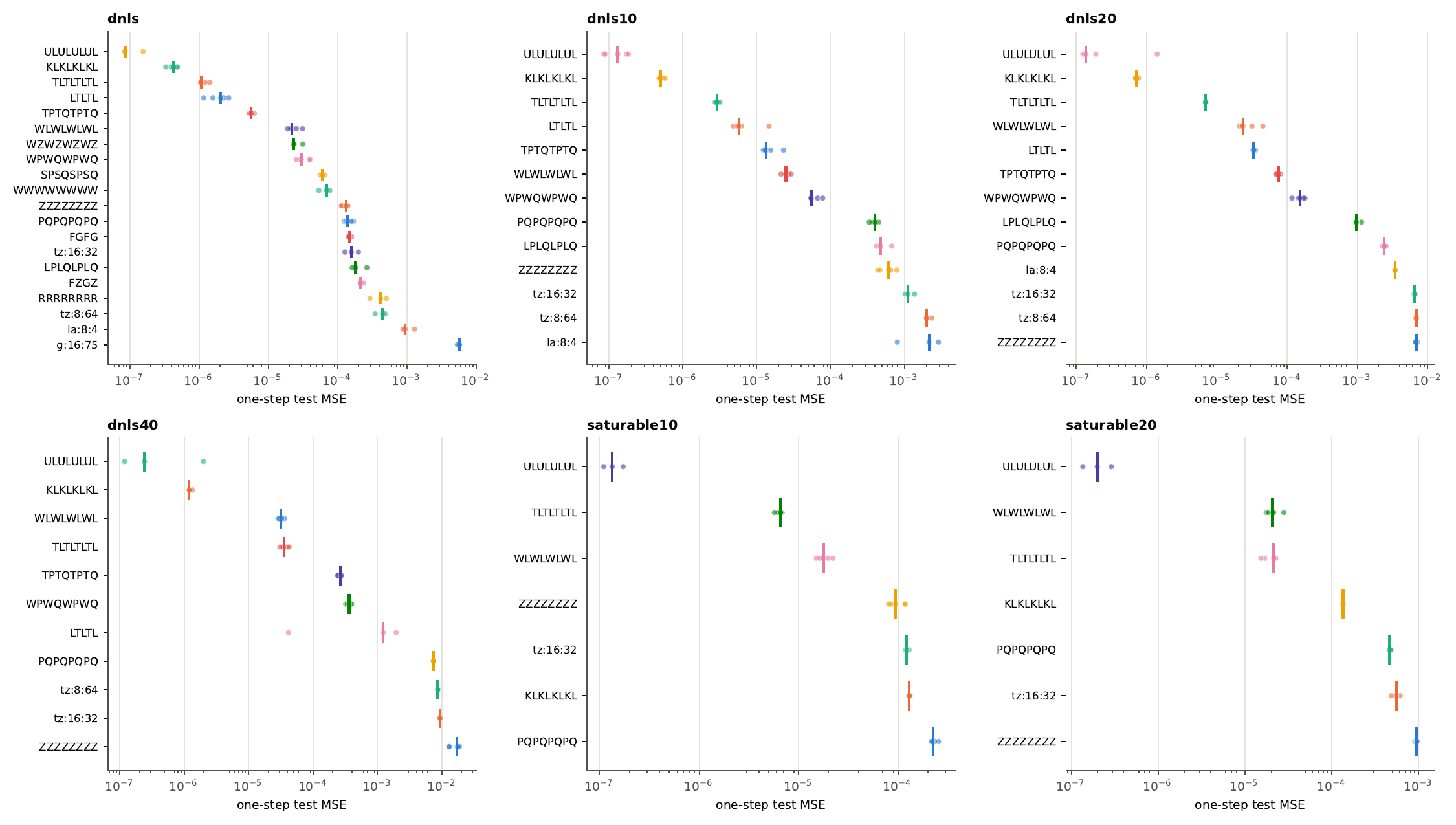}
\caption{Lattices (Table~\ref{tab:lattice}): every seed of every model.}
\label{fig:lattice-dots}
\end{figure}

\begin{figure}[p]
\centering
\includegraphics[width=\textwidth]{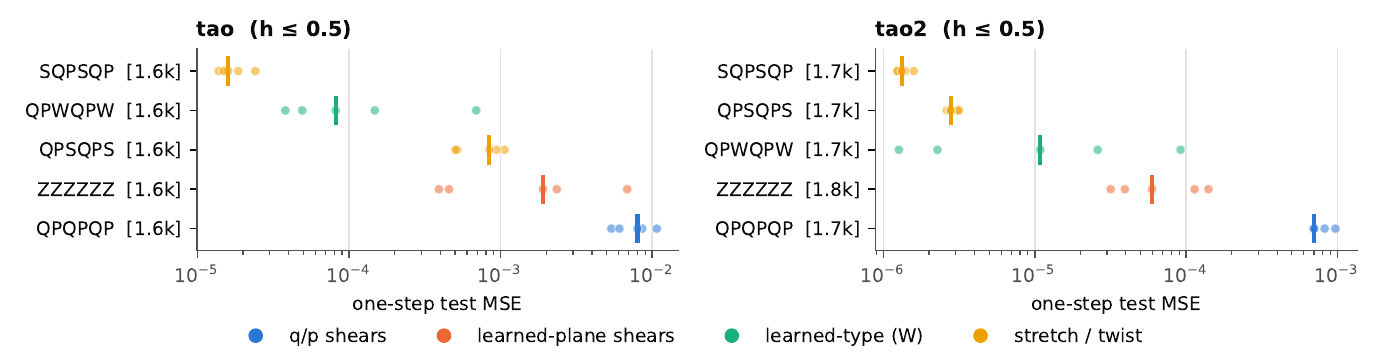}
\caption{Tao and coupled Tao, step size as input (Table~\ref{tab:stepped}): every seed.}
\end{figure}

\begin{figure}[p]
\centering
\includegraphics[width=\textwidth]{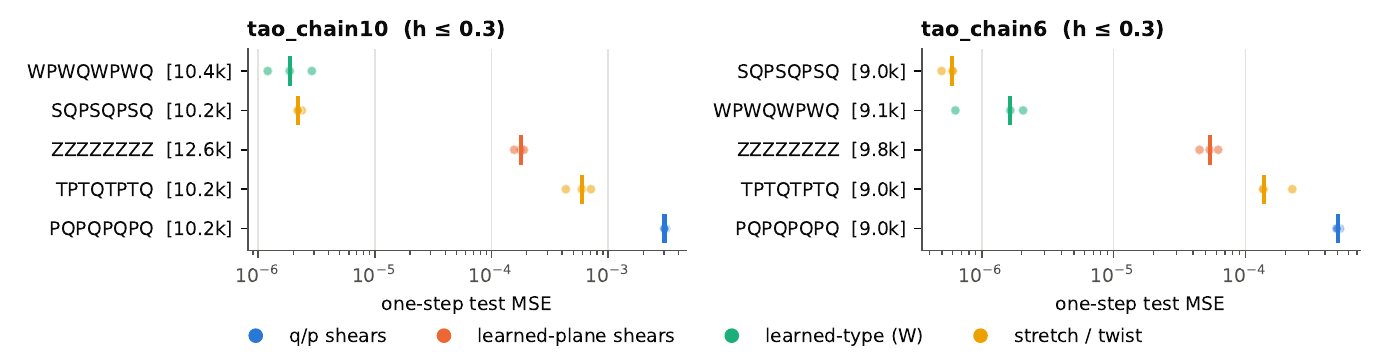}
\caption{Chains of Tao oscillators, step size as input (Table~\ref{tab:stepped}): every seed.
Rows are models; \texttt{TPTQTPTQ} is the twist network, which carries the wrong integral
for this system.}
\end{figure}

\begin{figure}[p]
\centering
\includegraphics[width=\textwidth]{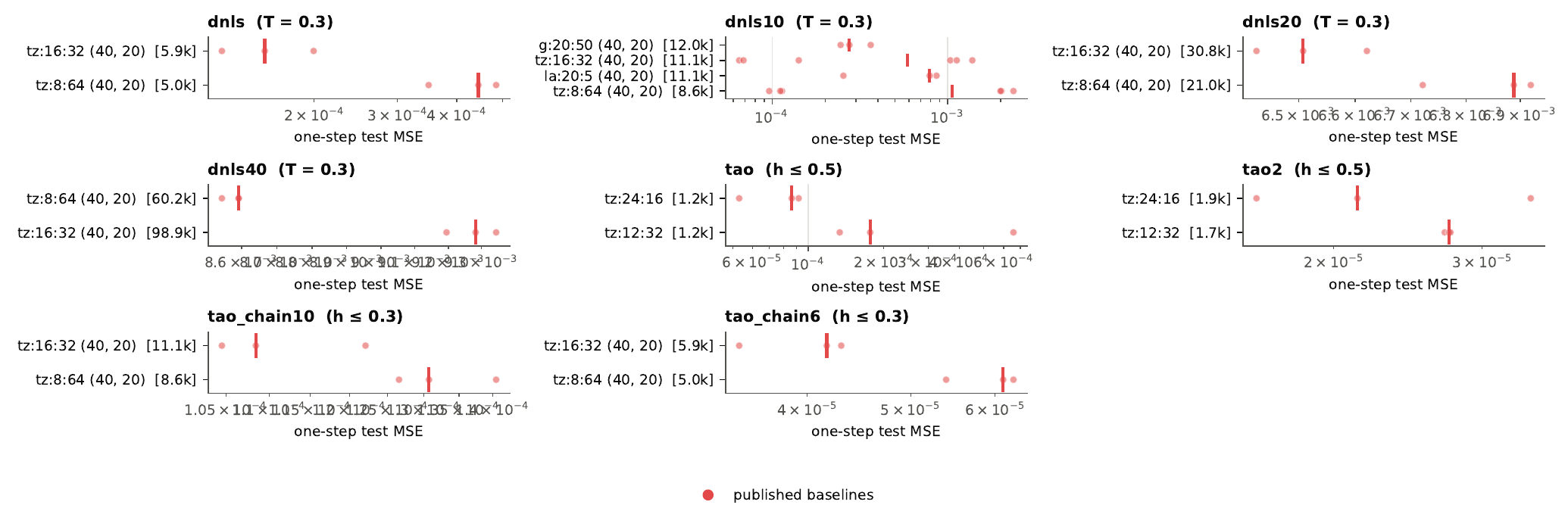}
\caption{Learned-plane shears in the configuration of the author's code
(\texttt{tz:layers:width}), on the lattices and the Tao-type systems, with the standard and
the four-fold training budget.}
\end{figure}

\begin{figure}[p]
\centering
\includegraphics[width=\textwidth]{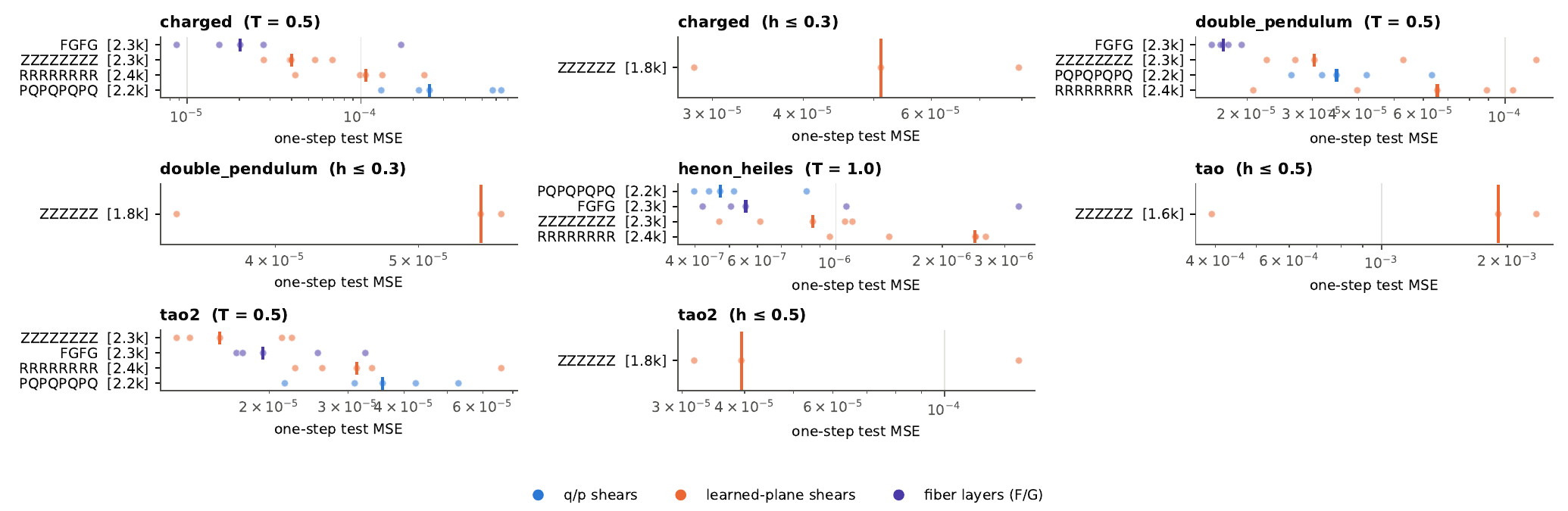}
\caption{Two degrees of freedom, fixed-time maps (Table~\ref{tab:twodof}) and step size as
input: every seed.}
\end{figure}

\subsection{Conjugated flows: reliability, type enumeration and the Toda lattice}

\begin{table}[h]
\centering\small
\caption{Reliability of conjugated flows over ten seeds: number of runs with test MSE below
$10^{-3}$. ``Fixed'' uses the correct type (twist for pendulum and Tao, stretch for the
saddle); ``learned'' uses the learned-type core from the standard or the near-parabolic
initialization of $M_j$; ``$\psi$ near identity'' starts the change of variables close to
the identity.}
\label{tab:reliability}
\begin{tabular}{lcccc}
\toprule
system & fixed core & learned, standard & learned, near-parabolic & learned, $\psi$ near identity\\
\midrule
pendulum & 10 & 4 & 6 & 7\\
Tao & 10 & 5 & 8 & 5\\
saddle & 8 & 9 & 5 & 9\\
\bottomrule
\end{tabular}
\end{table}

\begin{table}[h]
\centering\small
\caption{Saddle--center system: conjugated flows with the core's types fixed in advance,
five seeds each. The correct assignment (elliptic, hyperbolic) is separated from the others
by four orders of magnitude in training loss.}
\label{tab:enum}
\begin{tabular}{lcc}
\toprule
assignment & test MSE per seed & lowest training loss\\
\midrule
E, H (correct) & $5.5\cdot10^{-7}$, $8.2\cdot10^{-7}$, $1.5\cdot10^{-6}$, $2.0\cdot10^{-6}$, $6.2\cdot10^{-6}$ & $1.6\cdot10^{-7}$\\
E, E & $3.8\cdot10^{-3}$ -- $2.5\cdot10^{-2}$ & $1.7\cdot10^{-3}$\\
H, E & $9.9\cdot10^{-3}$ -- $3.1\cdot10^{-2}$ & $4.0\cdot10^{-3}$\\
H, H & $5.3\cdot10^{-3}$ -- $1.3\cdot10^{-2}$ & $1.3\cdot10^{-3}$\\
\bottomrule
\end{tabular}
\end{table}

\begin{table}[h]
\centering\small
\caption{Periodic Toda lattice with three particles, trained on $t\le2$: runs that fit (of
five) and RMSE at $t=10$ and $t=40$ for zero-momentum initial states, medians over the runs
that fit.}
\label{tab:toda}
\begin{tabular}{lccc}
\toprule
model & fits & $t=10$ & $t=40$\\
\midrule
conjugated, types fixed to E, E, P & 5 & $6.9\cdot10^{-3}$ & $2.8\cdot10^{-2}$\\
conjugated, types fixed to E, E, E & 3 & $7.2\cdot10^{-3}$ & $2.8\cdot10^{-2}$\\
conjugated, twist core & 3 & $7.1\cdot10^{-3}$ & $2.6\cdot10^{-2}$\\
conjugated, learned-type core & 1 & $6.8\cdot10^{-3}$ & $2.6\cdot10^{-2}$\\
stepped $q$/$p$ shears $\times8$, iterated & 5 & $1.8\cdot10^{-2}$ & $5.8\cdot10^{-2}$\\
stepped learned-plane shears $\times8$, iterated & 5 & $3.1\cdot10^{-2}$ & $2.1\cdot10^{-1}$\\
\bottomrule
\end{tabular}
\end{table}

\begin{figure}[p]
\centering
\includegraphics[width=0.85\textwidth]{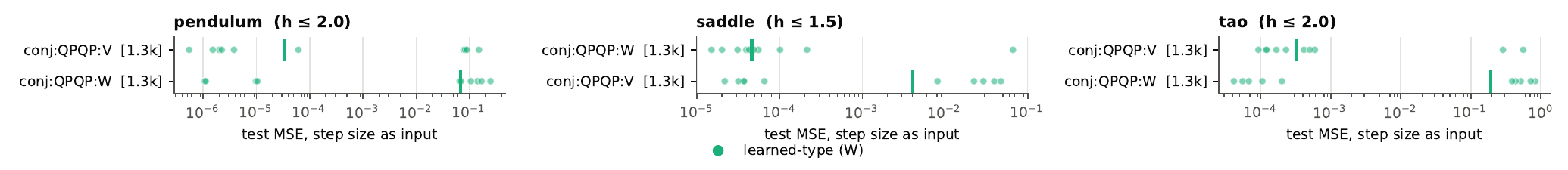}
\caption{Reliability sweep: test error of every seed of the conjugated learned-type model on
the pendulum, Tao's system and the saddle, for two initializations of $M_j$.}
\end{figure}

\begin{figure}[p]
\centering
\includegraphics[width=0.85\textwidth]{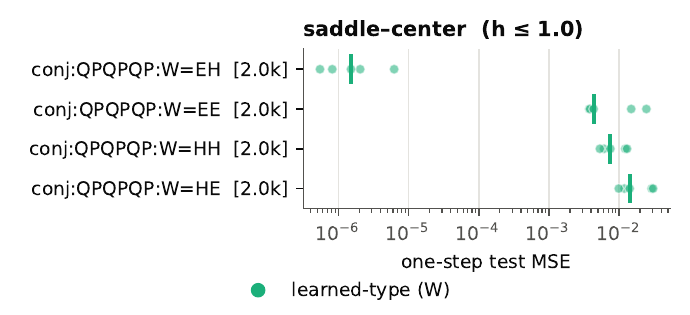}
\caption{Saddle--center system: every seed of the conjugated models with fixed types.}
\end{figure}

\begin{figure}[p]
\centering
\includegraphics[width=0.85\textwidth]{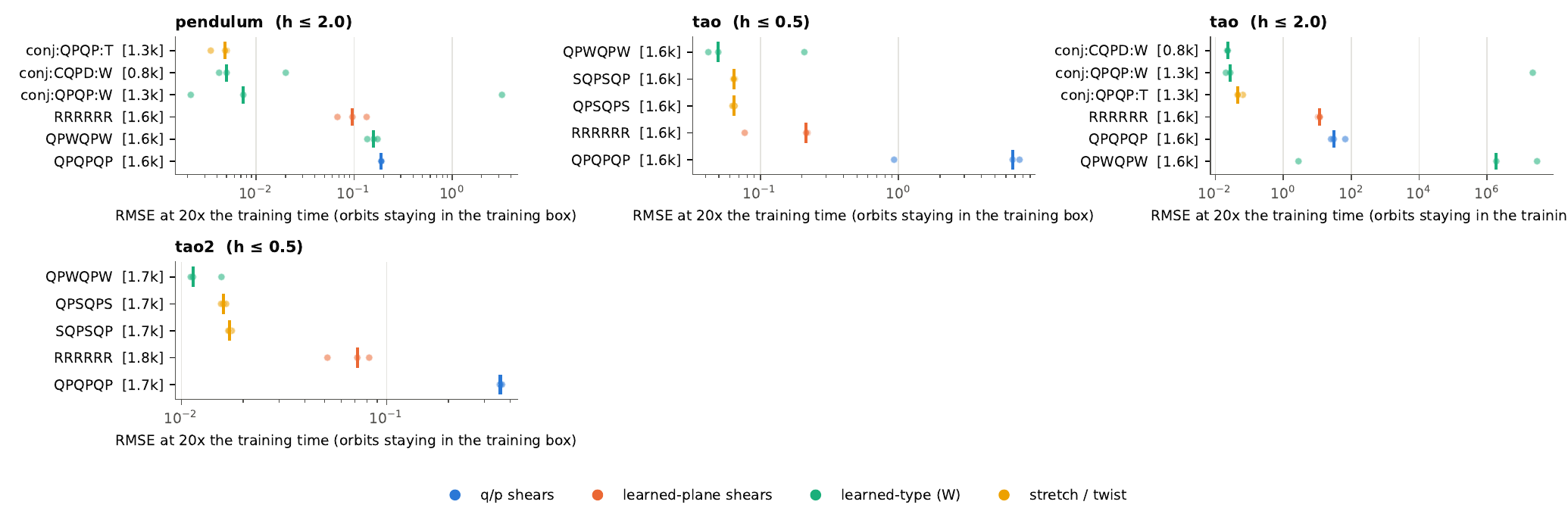}
\caption{Long-time prediction at twenty times the training horizon: conjugated flows (one
evaluation) against stepped networks (twenty steps), for orbits that stay in the training
box.}
\end{figure}

\subsection{Depth study and polar map: numbers}

\begin{table}[h]
\centering\small
\caption{Depth study (Figure~\ref{fig:depth}): median test MSE over two seeds.}
\label{tab:depth}
\begin{tabular}{llccccc}
\toprule
system & family & 2 & 3 & 4 & 6 & 8 layers\\
\midrule
pendulum & $q$/$p$ shears & $3.0\cdot10^{-2}$ & $4.6\cdot10^{-4}$ & $4.2\cdot10^{-7}$ & $2.3\cdot10^{-7}$ & $3.5\cdot10^{-7}$\\
 & learned-plane (ours) & $1.6\cdot10^{-4}$ & $2.3\cdot10^{-6}$ & $1.8\cdot10^{-6}$ & $7.1\cdot10^{-7}$ & $3.8\cdot10^{-7}$\\
 & learned-type & $1.5\cdot10^{-5}$ & $1.1\cdot10^{-6}$ & $4.2\cdot10^{-7}$ & $2.2\cdot10^{-7}$ & $2.5\cdot10^{-7}$\\
Tao, $T=0.5$ & $q$/$p$ shears & $6.6\cdot10^{-2}$ & $3.8\cdot10^{-2}$ & $3.7\cdot10^{-3}$ & $4.7\cdot10^{-6}$ & $7.0\cdot10^{-6}$\\
 & learned-plane (ours) & $2.8\cdot10^{-2}$ & $7.9\cdot10^{-3}$ & $1.4\cdot10^{-4}$ & $5.7\cdot10^{-5}$ & $1.1\cdot10^{-5}$\\
 & learned-type & $8.0\cdot10^{-3}$ & $5.3\cdot10^{-5}$ & $1.9\cdot10^{-5}$ & $1.5\cdot10^{-6}$ & $1.2\cdot10^{-6}$\\
charged particle & $q$/$p$ shears & $1.3\cdot10^{-1}$ & $6.3\cdot10^{-2}$ & $3.6\cdot10^{-3}$ & $3.4\cdot10^{-4}$ & $1.0\cdot10^{-4}$\\
 & learned-plane (ours) & $6.6\cdot10^{-3}$ & $8.8\cdot10^{-4}$ & $6.3\cdot10^{-4}$ & $1.4\cdot10^{-4}$ & $1.0\cdot10^{-4}$\\
 & learned-type & $9.3\cdot10^{-2}$ & $2.0\cdot10^{-3}$ & $4.7\cdot10^{-4}$ & $3.7\cdot10^{-4}$ & $1.0\cdot10^{-4}$\\
\bottomrule
\end{tabular}
\end{table}

\begin{table}[h]
\centering\small
\caption{Polar map of \cite{he2024deep}: test MSE, median of three seeds, five-layer models.}
\label{tab:polar}
\begin{tabular}{lcc}
\toprule
model & unit square & $[\tfrac12,\tfrac32]\times[0,\tfrac{3\pi}2]$\\
\midrule
\texttt{WPWQW} & $7.2\cdot10^{-6}$ & $1.9\cdot10^{-2}$\\
\texttt{PQWPQ} & $8.4\cdot10^{-6}$ & $2.3\cdot10^{-3}$\\
\texttt{CPQDP} (cotangent lifts) & $1.7\cdot10^{-5}$ & $2.6\cdot10^{-4}$\\
learned-plane (ours) & $2.2\cdot10^{-5}$ & $5.3\cdot10^{-4}$\\
\texttt{PQPQP} & $2.6\cdot10^{-5}$ & $2.5\cdot10^{-3}$\\
\texttt{PQSPQ} (network of \cite{he2024deep}) & $2.7\cdot10^{-5}$ & $2.7\cdot10^{-3}$\\
G-SympNet & $1.2\cdot10^{-4}$ & $3.3\cdot10^{-4}$\\
\bottomrule
\end{tabular}
\end{table}

\bibliographystyle{plain}
\bibliography{references}
\end{document}